\documentclass[11pt, leqno, twoside]{article}
\usepackage[left=3cm,top=3cm,right=3cm,bottom=4.5cm]{geometry}
\usepackage{amssymb, amsmath} 
\usepackage{epic} 
\usepackage{enumerate}
\usepackage{color}
\usepackage{mathtools}
\usepackage{amsthm}
\newtheorem{thm}{Theorem}

\newtheorem{lem}[thm]{Lemma}
\newtheorem{prop}[thm]{Proposition}
\newtheorem{defn}[thm]{Definition}

\DeclarePairedDelimiter\floor{\lfloor}{\rfloor}

\newcommand{\N}{\mathbb{N}}
\newcommand{\R}{\mathbb{R}}

\newcommand{\ve}{\varepsilon}

\newcommand{\Om}{\Omega}

\DeclareMathOperator{\dist}{dist}
\DeclareMathOperator{\diag}{diag}
\DeclareMathOperator{\diam}{diam}
\DeclareMathOperator{\rank}{rank}
\DeclareMathOperator{\supp}{supp}
\DeclareMathOperator{\sym}{sym}

\DeclareMathOperator{\id}{id}

\begin{document}

\title{Sobolev homeomorphisms with gradients of low rank via laminates}

\author{Daniel Faraco, Carlos Mora-Corral and Marcos Oliva
\\
\footnotesize Department of Mathematics, Faculty of Sciences, Universidad Aut\'onoma de Madrid, E-28049 Madrid, Spain\\
\footnotesize ICMAT CSIC-UAM-UCM-UC3M, E-28049 Madrid, Spain}

\date{\today}
\maketitle

\begin{abstract}
Let $\Omega\subset \R^{n}$ be a bounded open set. Given $2\leq m\leq n$, we construct a convex function $\phi :\Omega\to \R$ whose gradient  $f= \nabla \phi$ is a H\"older continuous homeomorphism, $f$ is the identity on $\partial \Omega$, the derivative $D f$ has rank $m-1$ a.e.\ in $\Omega$ and $D f$ is in the weak $L^{m}$ space $L^{m,w}$. The proof is based on convex integration and staircase laminates.
\end{abstract}

\section{Introduction}\label{sect: introduction}

Let $n \in \N$ and let $\Omega\subset \R^{n}$ be a bounded open set.
This paper deals with homeomorphisms $f:\Omega\to \R^{n}$ in the Sobolev class whose derivative $Df$ has rank less than $m$ at a.e.\ point, for a given $m \leq n$.

In \cite{Hencl11}, Hencl proves that there exists a homeomorphism $f$ in $W^{1,p}\left( (0,1)^{n},(0,1)^{n}\right)$, $1\leq p<n$, such that its Jacobian determinant $J_{f}$ equals zero a.e.
Notice that condition $p<n$ is necessary, since if $f \in W^{1,n}$  and $J_f \geq 0$, then $f$ satisfies the Luzin (N) condition and, consequently, the area formula holds. Therefore, any $f\in W^{1,n}$ with $J_{f}=0$ a.e.\ satisfies 
\[|f(\Om)|\leq \int_{\Om}J_f=0,\]
and, hence, $f$ cannot be a homeomorphism. In fact, in \cite{KaKoMa01} it is proved that if a Sobolev map $f$ is such that
\begin{equation}\label{lim Df n-ve=0}
\lim_{\ve\to 0}\int_{\Omega}| D f|^{n-\ve}=0,
\end{equation}
then $f$ satisfies condition (N), whereas the construction of \cite{Cerny11} elaborates on that of \cite{Hencl11} to show that there exists a homeomorphism $f\in W^{1,1}\left((0,1)^{n},\R^{n}\right)$ such that $J_{f}=0$ almost everywhere and $D f$ is in the grand Lebesgue space $L^{n)}$, i.e.,  
\[\sup_{0<\ve\leq n-1}\ve\int_{(0,1)^{n}}|D f|^{n-\ve}< \infty.\]
Obviously, such $f$ cannot satisfy (\ref{lim Df n-ve=0}).
 
The construction of Hencl \cite{Hencl11} has been further developed in \cite{DoHeSc14,Cerny15} to construct bi-Sobolev homeomorphisms $f$ with $J_f = 0$ almost everywhere and with zero minors of $Df$ almost everywhere.

All those constructions were based on a careful explicit construction and a limit process to obtain a Cantor set where the Jacobian is supported.  

In this work we explore a different way of obtaining such kind of pathological maps by using the staircase laminates invented by Faraco \cite{Faraco03}, in combination with the version of convex integration used by M\"uller and \v{S}ver\'{a}k \cite{MuSv03}. In fact, such laminates have turned out to be useful in a number of apparently unrelated problems such as $L^p$ theory of elliptic equations \cite{AsFaSz08}, Burkholder functions \cite{ BoSzVo13}, Hessians of rank-one convex functions \cite{CoFaMaMu05}, microstructure and phase transitions in solids \cite{Pedregal93, Muller96} and counterexamples of $L^{1}$ estimates \cite{CoFaMa05}. As in the case of Ornstein inequalities \cite{CoFaMa05}, laminates allow one to decouple the construction of pathological maps occurring in various situations into an analytical part and a geometrical part.  The analytical part is taken care
by the general theory of laminates and the version of convex integration based on in-approximations (in fact, in the problem at hand, on a slight evolution of the version for unbounded sets developed in 
\cite{AsFaSz08} which guarantees that the limit map is a homeomorphism). The geometrical part, which is the key in the whole process, consists in finding a suitable staircase laminate. In fact, in \cite{Faraco06} it was sketched how to use laminates to obtain, in dimension $2$, a convex function whose gradient $f$ is in $W^{1,p}$ for all $p<2$, and satisfies $J_{f}=0$ a.e., recovering another interesting example of Alberti and Ambrosio \cite{AlAm99}. 

We mention that gradients of convex functions are interesting in its own right; for example, they are the key ingredient in the polar decomposition of the Brenier map in mass transportation \cite{Brenier91}.  In this regard, our example show some limits of the regularity for transport maps \cite{DeFi14}.

In the current work we show that with staircase laminates it is possible to combine the results of \cite{Faraco06,AlAm99,Hencl11}. In fact, we also recover the result
of \v{C}ern\'{y} \cite{Cerny15}, where he constructs a homeomorphism $f$ with all its minors of $m$-th order equal to zero almost everywhere belonging to $W^{1,p}$ with $1\leq p <\frac{n}{n+1-m}$.

To be precise, we build a probability measure formed by  staircase laminates in the planes parallel to the axes, which can be pushed to show that not only the Jacobian is zero but also that $Df$ has all its minors of order $m$ equal to zero almost everywhere and $Df$ is in $L^{m,w}$, i.e., there exists a constant $C>0$ such that
\[ |\lbrace x\in\Omega: |D f (x)|>t \rbrace|\leq C t^{-m},\qquad t>0.\]
In the particular case when $m=n$ we have $L^{n,w}\subset L^{n)}$, so our result is sharp in the sense explained before. In fact, it seems likely that we could push our construction so that $Df(x)$ is diagonal for all $x$ except in a set of arbitrarily small measure, but we do not pursue this issue here.

This is the theorem that we will prove.
  
  \begin{thm}\label{th1}
Let $\Om\subset \R^n$ be a bounded open set, $m\in \N$, $2\leq m\leq n$, $\delta>0$ and $\alpha\in (0,1)$. Then there exists a convex function $u:\Omega\to \R$ such that its gradient $f=\nabla u$ satisfies:

\begin{enumerate}[i)]
\item $f \in W^{1,1} (\Omega, \R^n)$ and $f:\Omega\to\Omega$ is a homeomorphism.
\item $f=\id$ on $\partial \Om$.
\item $\rank (D f(x))<m$ for a.e.\ $x\in\Om$.
\item $D f\in L^{m,w}\left(\Omega,\R^{n\times n}\right)$.
\item\label{item:theorem5} $\|f-\id\|_{C^{\alpha}(\overline{\Omega})}<\delta$ and $\|f^{-1}-\id\|_{C^{\alpha}(\overline{\Omega})}<\delta$.
\end{enumerate}
\end{thm}

Our theorem shows that a H\"older continuous Brenier map (the gradient of a convex function; see, e.g., \cite[p.\ 67]{Villani03} for the definition) may have a rather pathological behaviour. 

As discussed at the beginning, the result is known to be sharp in the sense of integrability in the case $m=n$. Notice that in this case, as explained in \cite{Hencl11}, using the area formula for Sobolev mappings (\cite{Hajlasz93}) we have that this kind of homeomorphisms sends a set of full measure to a null set, and a null set to a set of full measure, i.e., there exists $ Z\subset \Omega$ of measure zero such that
\[ |f\left(\Omega\setminus Z\right)|=\int_{f\left(\Omega\setminus Z\right)} dy=\int_{\Omega\setminus Z} J_{f}(x)\,dx=0\]
 and
 \[ |f\left( Z\right)|=|f\left( \Omega\right)|-|f\left( \Omega\setminus Z\right)|=|f\left( \Omega\right)|.\]

The structure of the paper is as follows.

In Section \ref{sect: sketch of the proof} we introduce the concept of laminate of finite order and  sketch the construction of the laminate that will be central in the proof of Theorem \ref{th1}. Notice that if $m\neq n \neq 2$ the construction is genuinely different from previous staircase laminates as we need to lower the rank accordingly. 

Section \ref{sect: notation} presents the general notation of the paper.

In Section \ref{sect: approximation of laminates by functions} we prove that, given a laminate of finite order, there exists a function $f$ whose derivative is close to the laminate. Moreover, if the laminate is supported in the set of positive definite matrices then $f$ is a homeomorphism.

Section \ref{se:cutting} shows that if we modify a H\"older homeomorphism by cutting and pasting, the map obtained is still a H\"older homeomorphism .

Section \ref{sect: lemmas}, which is the bulk of the paper, constructs a sequence of laminates that converges to the probability measure sketched in Section \ref{sect: sketch of the proof}, as well as a sequence of functions that approximate the laminates.

In Section \ref{sect: proof of the theorem} we prove Theorem \ref{th1}. 

In Section \ref{sect: sharpness} we show the sharpness of our result by proving that there does not exist a H\"{o}lder continuous homeomorphism in $W^{1,m}(\Omega,\R^n)$ such that $f=\id$ on $\partial \Omega$ and $\rank(Df)<m$ a.e.\ in $\Omega$.
Moreover, H\"{o}lder continuity can be dispensed with if $f \in W^{1,p}(\Omega,\R^n)$ for somw $p > m$.

\emph{Note:} When our paper was presented, S. Hencl called our attention to a recent preprint by Liu and Mal\'y \cite{LiMa16},
where a result very similar to our Theorem~\ref{th1} was proved, with a construction related to laminates but not inspired in \cite{Faraco06}. It would be very interesting to see how much these examples have in common. For instance, an understanding of the support of the distributional Jacobian or, in general, the distributional minors is pending.
The advantage of the method presented in this paper is that, once an extremal staircase laminate is found, which is a relatively fast task (Section~\ref{sect: sketch of the proof}), the extremal mapping quite likely will appear, for example by the in-approximation method.  Notice that this last step is not always possible, as shown by the case of monotone maps (the 
staircase laminate from \cite{CoFaMa05} is supported in the range of gradients of planar monotone maps, which are regular by \cite{AlAm99}). 
Finally, we mention that it will also be 
worthy to see the relation with the works \cite{KiKr11,KiKr16}, where the results of \cite{CoFaMa05} are easily  recovered from the study of rank-one convex functions that are one-homogeneous.

\section{Sketch of the proof}\label{sect: sketch of the proof}

The next definition introduces the concept of laminate of finite order \cite{Dacorogna89, Pedregal93, MuSv03, AsFaSz08}. 

\begin{defn}\label{de:laminate}

The family $\mathcal{L}(\R^{n\times n})$ of laminates of finite order is the smallest family of probability measures in $\R^{n\times n}$ with the properties:

\begin{itemize}
\item $\mathcal{L}(\R^{n\times n})$ contains all the Dirac masses.
\item If $\sum_{i=1}^{N}\lambda_i\delta_{A_i}\in \mathcal{L}(\R^{n\times n})$ and $A_N=\lambda B+(1-\lambda)C$, where $\lambda\in [0,1]$ and $\rank \,(B-C)=1$, then the probability measure 
\[\sum_{i=1}^{N-1}\lambda_i\delta_{A_i}+\lambda_N(\lambda \delta_B+(1-\lambda)\delta_C)\]
is also in $\mathcal{L}(\R^{n\times n})$. 
\end{itemize}
\end{defn}

Note that any laminate of finite order is a convex combination of Dirac masses.
Since in this work we will only use laminates of finite order, for simplicity they will be just called \emph{laminates}.

In this section we construct the sequence of laminates $ \nu_{k}$ of finite order that is behind the whole article.
The actual proof will consist in approximating $\nu_{k}$ with laminates of finite order supported in the set of positive definite matrices, then use Proposition \ref{laminate-homeo} to obtain homeomorphisms that are close to the approximate laminates in small regions of the domain, then paste the obtained homeomorphisms to construct a homeomorphism in the whole domain and, finally, a limit passage will yield the homeomorphism $f$ of Theorem \ref{th1}.
The fact that $f$ is the gradient of a convex function is standard since $Df$ was constructed to be symmetric positive semidefinite.

Although this section is not necessary for the proof of Theorem \ref{th1}, it will help the reader to follow the construction of Section \ref{sect: lemmas}.

In order to construct $\nu_k$, we need to define the sets 
\[ S^{k}_{i}=\left\lbrace A=k\diag(v): v\in \lbrace 0,1\rbrace^{n}\text{ and }\rank\left(A\right)=n-i  \right\rbrace, \qquad k\in\N, \quad i\in \{0,\ldots,n-m\}\]
and 
\[E=\{ A\in \R^{n\times n}: \rank (A)<m\}.\]
Thus, the matrices of $S_{i}^{k}$ are $k$ times the identity matrix in the $(n-i)$-dimensional linear subspaces parallel to the axes.

The main property of the laminates to be constructed is as follows: for each $k\in\N$, $\nu_{k}$ is supported in $\bigcup_{i=0}^{n-m}S^{k}_{i}\cup E $, 
\[ |A|\leq k \text{ for all } A\in \supp \left(\nu_{k}\right),\]
and
\[\nu_{k}\left(S^{k}_{i}\right)\leq C k^{i-n},\]
for some $C>0$.

The weak$^*$ limit $\nu$ of $\nu_k$ is supported in the set $E$. It satisfies that there exists a constant $C>0$ such that
\begin{equation}\label{eq:sketch nu}
 \nu\left(\{ |A|>t \}\right)\leq C t^{-m},\qquad t>0.
\end{equation}
This last inequality will give us the desired integrability of the derivative of the homeomorphism.

The laminates $\nu_{k}$ are defined inductively as follows. We start with $\nu_{1}=\delta_{I}$.
Now, given 
\begin{equation}\label{nu_k-1}
\nu_{k-1}=\sum_{j=1}^{N}\lambda_j\delta_{A_j}\in \mathcal{L}(\R^{n\times n}),
\end{equation}
 with $A_{j}\in \bigcup_{i=0}^{n-m}S^{k-1}_{i}\cup E$, all different, we are going to split the matrices of $S_{i}^{k-1}$ in matrices in $\bigcup_{i=0}^{n-m}S^{k}_{i}\cup E $ following rank-one lines as in Definition \ref{de:laminate}.

Let $A\in \supp \left(\nu_{k-1}\right)$. If $A\in E$ we define $\nu_{A}=\delta_{A}$, whereas if $A\in S^{k-1}_{i}$ for some $i\in\{0,\ldots,n-m\}$, we construct $\nu_{A}$ inductively. Without loss of generality, 
\[A=(k-1)\diag(\underbrace{1,\ldots,1}_{n-i},\underbrace{0,\ldots,0}_{i}).\]
We shall construct families
\begin{equation}\label{eq:familiesBl}
 \{B_{\ell,j}\}_{\substack{\ell=0,\ldots,n-i \\ j=0,\ldots, 2^{\ell}-1}} \subset \R^{n\times n} \quad \text{and} \quad \{\lambda_{\ell,j}\}_{\substack{\ell=0,\ldots,n-i \\ j=0,\ldots, 2^{\ell}-1}} \subset [0,1]
\end{equation}
by finite induction on $\ell$.

Let $B_{0,0}=A$, $\lambda_{0,0}=1$ and for $0\leq \ell\leq   n-i-1$,  $0\leq j\leq 2^{\ell}-1$, we assume that $\{B_{\ell,j}\}_{j=0}^{2^{\ell}-1}$ and $\{\lambda_{\ell,j}\}_{j=0}^{2^{\ell}-1}$ have been defined, $B_{\ell,j}$ are diagonal, $\lambda_{\ell,j}\geq 0$,
\begin{equation}\label{eq:l=1,B00}
 \sum_{j=0}^{2^{\ell}-1}\lambda_{\ell,j}=1,\qquad B_{0,0}=\sum_{j=0}^{2^{\ell}-1}\lambda_{\ell,j}B_{\ell,j} ,\end{equation}
\begin{equation}\label{sketch proof Blj es laminado}
\sum_{j=0}^{2^{\ell}-1}\lambda_{\ell,j}\delta_{B_{\ell,j}}\in\mathcal{L}(\R^{n\times n})
\end{equation}
and
\begin{equation}\label{sketch proof valores Blj mayores que l}
\left(B_{\ell,j}\right)_{\alpha,\alpha}=\begin{cases}
k-1 & \text{if } \alpha=\ell+1,\ldots, n-i,\\
0 & \text{if } \alpha=n-i+1,\ldots, n.
\end{cases}
\end{equation}
We also assume that if $B_{\ell,j}\notin E$ then
\begin{equation}\label{sketch proof valores Blj menores que l}
\left(B_{\ell,j}\right)_{\alpha,\alpha}\in \{0,k\}, \qquad \alpha=1,\ldots, \ell,
\end{equation}
\begin{equation}\label{sketch proof valor lambdalj}
\lambda_{\ell,j}= \frac{(k-1)^{\beta_{\ell,j}}}{k^{\ell}},
\end{equation}
where
\[\beta_{\ell,j}:=\#\{\alpha\in\{1,\ldots,\ell\}: \left(B_{\ell,j}\right)_{\alpha,\alpha}=k\}=\rank(B_{\ell,j})-n+i+\ell,\]
and for each $B_{\ell,j'}\notin E$ with $j'\neq j$, we have $B_{\ell,j'}\neq B_{\ell,j}$.

With the above induction hypotheses, we construct $\{B_{\ell+1,j}\}_{j=0}^{2^{\ell+1}-1}$ and $\{\lambda_{\ell+1,j}\}_{j=0}^{2^{\ell+1}-1}$ as follows.
 For any $0\leq j\leq 2^{\ell}-1$, we define
\[ B_{\ell+1,2j}=\begin{cases}
B_{\ell,j}-\diag\left(\underbrace{0,\ldots,0}_{\ell}, k-1,\underbrace{0,\ldots,0}_{n-\ell-1}  \right), & \text{if }  B_{\ell,j}\notin E,\\
B_{\ell,j}, & \text{if } B_{\ell,j}\in E,
\end{cases}\]
\[ B_{\ell+1,2j+1}=\begin{cases}
B_{\ell,j}+\diag\left(\underbrace{0,\ldots,0}_{\ell},1,\underbrace{0,\ldots,0}_{n-\ell-1}  \right), & \text{if } B_{\ell,j}\notin E,\\
B_{\ell,j}, & \text{if } B_{\ell,j}\in E ,
\end{cases}\]
\[ \lambda_{\ell+1,2j}=\lambda_{\ell,j}\frac{1}{k}\quad\text{and}\quad\lambda_{\ell+1,2j+1}=\lambda_{\ell,j}\frac{k-1}{k}.\]
Now we check the induction hypotheses.

We have $B_{\ell,j}=\frac{1}{k}B_{\ell+1,2j}+\frac{k-1}{k}B_{\ell+1,2j+1}$ and $\rank(B_{\ell+1,2j}-B_{\ell+1,2j+1})\leq 1$, so property (\ref{sketch proof Blj es laminado}) holds for $\ell+1$.
Analogously, equalities \eqref{eq:l=1,B00} are easily seen to hold for $\ell+1$ as well.

Now, let $0\leq j\leq 2^{\ell+1}-1$.
We have
\[\left(B_{\ell+1,j}\right)_{\alpha,\alpha}=\left(B_{\ell,\floor{\frac{j}{2}}}\right)_{\alpha,\alpha},\qquad \alpha\neq \ell+1\]
and
\[\left( B_{\ell+1,j}\right)_{\ell+1,\ell+1}=\begin{cases}
0, & \text{if } B_{\ell,\floor{\frac{j}{2}}}\notin E,\ j \text{ even},\\
k, & \text{if } B_{\ell,\floor{\frac{j}{2}}}\notin E,\ j \text{ odd},\\
\left( B_{\ell,\floor{\frac{j}{2}}}\right)_{\ell+1,\ell+1}, & \text{if } B_{\ell,\floor{\frac{j}{2}}}\in E.
\end{cases}\]
Therefore, equality \eqref{sketch proof valores Blj mayores que l} holds for $\ell+1$.

Now fix $\ell,j$ such that $B_{\ell+1,j}\notin E$.
Then $B_{\ell,\floor{\frac{j}{2}}}\notin E$ and, hence, property \eqref{sketch proof valores Blj menores que l} holds for $\ell+1$.
We also have
\[
 \beta_{\ell+1,j} = \begin{cases}
 \beta_{\ell,\frac{j}{2}}, & j \text{ even} , \\
 \beta_{\ell,\frac{j-1}{2}}+1 , & j \text{ odd} ,
 \end{cases}
\]
so \eqref{sketch proof valor lambdalj} holds for $\ell+1$.
Finally, let $j'\neq j$ be such that $B_{\ell+1,j'}\notin E$. If $\floor{\frac{j}{2}}\neq \floor{\frac{j'}{2}}$, then $B_{\ell,\floor{\frac{j'}{2}}}\neq B_{\ell,\floor{\frac{j}{2}}}$, and, hence, 
$B_{\ell+1,j'}\neq B_{\ell+1,j}$, 
 whereas if $\floor{\frac{j}{2}}= \floor{\frac{j'}{2}}$, then $(B_{\ell+1,j'})_{\ell+1,\ell+1}\neq (B_{\ell+1,j})_{\ell+1,\ell+1}$, and, hence,
$B_{\ell+1,j'}\neq B_{\ell+1,j}$.
Here ends the inductive construction of the families \eqref{eq:familiesBl} with the required properties.

Thanks to (\ref{sketch proof valores Blj mayores que l}) and (\ref{sketch proof valores Blj menores que l}) we have, for all $0\leq j\leq 2^{n-i}-1$, 
\begin{equation}\label{sketch proof Bn-ij en Skl union E}
B_{n-i,j}\in \bigcup_{\ell=i}^{n-m}S^{k}_{\ell}\cup E
\end{equation}
whereas (\ref{sketch proof valor lambdalj}) yields
\begin{equation}\label{sketch proof valor lambdan-ij}
\lambda_{n-i,j}=\frac{(k-1)^{\rank\left(B_{n-i,j}\right)}}{k^{n-i}}.
\end{equation}
We define 
\begin{equation}\label{nu_A}
\nu_{A}=\sum_{j=0}^{2^{n-i}-1}\lambda_{n-i,j}\delta_{B_{n-i,j}},
\end{equation}
which is a laminate due to (\ref{sketch proof Blj es laminado}).

 From (\ref{sketch proof Bn-ij en Skl union E}) we get
 \begin{equation}\label{nu_A(Skl)=0}
 \nu_{A}\left(\bigcup_{\ell=0}^{i-1}S^{k}_{\ell}\right)=0,
 \end{equation}
 whereas for $\ell\in\{i,\ldots,n-m\}$, we have, due to (\ref{sketch proof valores Blj mayores que l}) and (\ref{sketch proof valor lambdan-ij})
 \begin{equation}\label{nu(Skl)}
 \nu_{A}\left(S^{k}_{\ell}\right)=\sum_{j:B_{n-i,j}\in S^k_{\ell}}\lambda_{n-i,j} = \sum_{j:B_{n-i,j}\in S^k_{\ell}} \frac{(k-1)^{n-\ell}}{k^{n-i}}\leq \binom{n-i}{n-\ell}\frac{(k-1)^{n-\ell}}{k^{n-i}} ,
 \end{equation}
since the $B_{n-i,j}$ ($0\leq j\leq 2^{n-i}-1$) not in $E$ are all different.
Thus, for each $ j\in \{1,\ldots,N\}$, given $A_j$ appearing in (\ref{nu_k-1}), we have constructed $\nu_{A_{j}}$ as in (\ref{nu_A}) if $A_j\notin E$ and $\nu_{A_j}=\delta_{A_j}$ if $A_j\in E$. So 
\begin{equation}\label{nu A_j}
\nu_{A_{j}}\left(S^k_{\ell}\right)=0 \quad \text{if } A_{j}\in E,\qquad \forall \ell\in\{0,\ldots,n-m\}
\end{equation}
 and we define the probability measure $\nu_{k}:=\sum_{j=1}^{N}\lambda_j\nu_{A_j}$, which is supported in $\bigcup_{i=0}^{n-m}S^{k}_{i}\cup E$ by (\ref{sketch proof Bn-ij en Skl union E}).
 In fact, $\nu_{k}$ is a laminate, but this is not important on the proof.
We observe that we have proved
\begin{equation}\label{sketch proof norma A en nuk}
|A|\leq k \text{ for all } A\in \supp \left(\nu_{k}\right).
\end{equation}
We also let
\begin{equation}\label{C_k}
C_{k}=\prod_{j=1}^{k-1}\left( 1+2^{n}j^{-2}\right)
\end{equation}
and we will prove that for all $k\in \N$ and $i=0,\ldots,n-m$,
\begin{equation}\label{sketch proof cota nuk}
 \nu_{k}\left(S^{k}_{i}\right)\leq C_{k} k^{i-n}.
\end{equation}
We proceed by induction. The inequality for $k=1$ is immediate since $\nu_1=\delta_I$. Let $k\geq 2$ and suppose that for $i=0,\ldots,n-m$, inequality \eqref{sketch proof cota nuk} holds for $k-1$.
Since the $A_j$ of (\ref{nu_k-1}) are different, we have that
$\nu_{k-1}(A_j)=\lambda_j$.
Now, for all $\ell\in\{0,\ldots,n-m\}$, we use (\ref{nu_A(Skl)=0}) and (\ref{nu A_j}) to get
\begin{align*}
\nu_{k}\left( S^{k}_{\ell}\right)&=\sum_{j=1}^{N}\nu_{k-1}\left( A_{j}\right)\nu_{A_{j}}\left(S^{k}_{\ell}\right)= \left[ \sum_{j:A_{j}\in E} + \sum_{i=0}^{n-m}\sum_{j:A_{j}\in S^{k-1}_{i}} \right] \nu_{k-1}\left( A_{j}\right)\nu_{A_{j}}\left(S^{k}_{\ell}\right) \\
&= \sum_{i=0}^{\ell}\sum_{j:A_{j}\in S^{k-1}_{i}}\nu_{k-1}\left( A_{j}\right)\nu_{A_{j}}\left(S^{k}_{\ell}\right).
\end{align*}
We use that \eqref{sketch proof cota nuk} is valid for $k-1$ to get that for all $i\in \{0,\ldots,n-m\}$,
\[
\frac{(k-1)^{n-\ell}}{k^{n-i}} \sum_{j:A_{j}\in S^{k-1}_{i}}\nu_{k-1}\left( A_{j}\right) = \frac{(k-1)^{n-\ell}}{k^{n-i}}\nu_{k-1}\left(S^{k-1}_{i}\right) \leq C_{k-1} \frac{(k-1)^{i-\ell}}{k^{n-i}}.
\]
In addition,
\[\sum_{i=0}^{\ell}\binom{n-i}{n-\ell}\left((k-1)k\right)^{i-\ell}\leq 1+\sum_{i=0}^{\ell-1}\binom{n-i}{n-\ell}\left(k-1\right)^{2(i-\ell)}\leq  1+2^{n}(k-1)^{-2},\]
where we have used the crude inequality $\sum_{i=0}^{\ell-1}\binom{n-i}{n-\ell} \leq 2^n$.
We combine the last three inequalities and (\ref{nu(Skl)}) to get
\begin{align*}
\nu_{k}\left( S^{k}_{\ell}\right) & = \sum_{i=0}^{\ell}\sum_{j:A_{j}\in S^{k-1}_{i}}\nu_{k-1}\left( A_{j}\right)\nu_{A_{j}}\left(S^{k}_{\ell}\right)\leq \sum_{i=0}^{\ell}\sum_{j:A_{j}\in S^{k-1}_{i}}\nu_{k-1}\left( A_{j}\right)\binom{n-i}{n-\ell}\frac{(k-1)^{n-\ell}}{k^{n-i}}\\
&\leq \sum_{i=0}^{\ell}\binom{n-i}{n-\ell}C_{k-1} \frac{(k-1)^{i-\ell}}{k^{n-i}}\leq k^{\ell-n}C_{k-1}\left( 1+2^{n}(k-1)^{-2}\right)=C_{k}k^{\ell-n} ,
\end{align*}
which proves \eqref{sketch proof cota nuk}.
We observe from (\ref{C_k}) that $\lim_{k\to\infty}C_{k}<\infty$. Consequently,
\begin{equation}\label{sketch proof nukSk}
\nu_{k}\left(\bigcup_{i=0}^{n-m}S^{k}_{i}\right)\leq \sum_{i=0}^{n-m}C_{k}k^{i-n}\leq Ck^{-m}
\end{equation}
for some $C>0$.

Next, we will prove by induction that exists $\tilde{C}>0$ such that for all $k\in\N$,
\begin{equation}\label{sketch proof integrabilidad nuk}
\nu_{k}\left(\{ A \in \R^{n \times n} : |A|>t \}\right)\leq \tilde{C}t^{-m},\qquad t>0.
\end{equation}
For simplicity of notation, the set $\{ A \in \R^{n \times n} : |A|>t \}$ will be abbreviated as $\{ |A|>t \}$.
When $k=1$, we have $\nu_1=\delta_I$, hence, (\ref{sketch proof integrabilidad nuk}) is obvious. Now, we divide the inductive step in three cases, according to the values of $t$.

If $t \geq k+1$, we use (\ref{sketch proof norma A en nuk}) to obtain that $\nu_{k+1}\left(\{ |A|>t \}\right)=0$.

In the case $t< k$, we first show that if $|A|>t$, $A_j\in \supp (\nu_k)$ and $\nu_{A_j}(A)>0$, then $|A_j|>t$.
Indeed, if $A_j\in E$, then $\nu_{A_j}=\delta_{A_j}$, so, as $\nu_{A_j}(A)>0$, we have $A=A_j$, and therefore $|A_j|>t$; whereas if $A_j\in \bigcup_{i=0}^{n-m}S_{i}^{k}$, then $|A_j|=k>t$.
Therefore, if $|A_j|\leq t$, then 
$\nu_{A_j}\left(\{ |A|>t \}\right)=0$,
and, hence,
\begin{align*}
\nu_{k+1} \left(\{ |A|>t \}\right) & = \sum_{\substack{j:A_j\in \supp (\nu_k) \\ |A_j|>t}}\nu_{k}(A_j)\nu_{A_j}\left(\{ |A|>t \}\right)\\
&\leq\nu_{k}\left(\left\lbrace A_j\in \supp (\nu_k): |A_j|>t\right\rbrace\right)\leq \nu_{k}\left(\{ |A|>t \}\right)\leq \tilde{C}t^{-m}.
\end{align*}

Finally, in the case $k\leq t < k+1$, we use that if $A\in\supp(\nu_{k+1})$ and $|A|>t$, then, by (\ref{sketch proof norma A en nuk}), we have that for all $A_j \in\supp(\nu_{k})\cap E$, equalities $\nu_{A_j}(A)=\delta_{A_j}(A)=0$ are satisfied. Hence, thanks to \eqref{sketch proof nukSk},
\[
 \nu_{k+1} \left(\{ |A|>t \}\right) = \sum_{\substack{j:A_j\in \supp (\nu_k) \\ A_j \notin E}} \nu_{k}(A_j)\nu_{A_j}\left(\{ |A|>t \}\right) \leq \nu_{k}\left(\bigcup_{i=0}^{n-m}S^{k}_{i}\right)\leq Ck^{-m}\leq \tilde{C}t^{-m}.
\]
This finishes the proof of \eqref{sketch proof integrabilidad nuk}.

Let $\nu$ be the weak$^*$ limit of $\nu_k$ as $k\to\infty$.
Thanks to \eqref{sketch proof nukSk}, $\nu$ is supported in $E$ and, by (\ref{sketch proof integrabilidad nuk}), inequality \eqref{eq:sketch nu} holds.

We will try to illustrate this construction with some pictures.

In the simplest case $n=m=2$, $E$ is the set of matrices with zero determinant, $\nu_{1}=\delta_{I}$ and the first steps of the construction are depicted in Figures \ref{n=2 first split} and \ref{n=2 second split}.
In the first step (Figure \ref{n=2 first split}) we split $B_{0,0}=I$ into $B_{1,0}=\diag(0,1)$ and $B_{1,1}=\diag\left(2,1\right)$.
\begin{figure}[h]
\begin{picture}(180,110)(0,100)

\put(50,100){\vector(0,1){100}}
\put(50,100){\vector(1,0){100}}

\dashline{10}(50,100)(150,200)

\put(80,130){\circle{10}}\put(80,145){\makebox(0,0){$B_{0,0}$}}
\put(50,130){\circle*{10}}
\put(35,120){\makebox(0,0){$B_{1,0}$}}
\put(110,130){\circle*{10}}
\put(125,120){\makebox(0,0){$B_{1,1}$}}
\put(50,130){\line(1,0){60}}
\end{picture}
\caption{\label{n=2 first split} First split, $n=m=2$, $B_{0,0}=I$: $B_{1,0}=\diag(0,1)$, $B_{1,1}=\diag\left(2,1\right)$.}
\end{figure}
As $B_{1,0}\in E$, the second step (Figure \ref{n=2 second split}) consists in splitting $B_{1,1}$ into $B_{2,2}=\diag\left(2,0\right)$ and $B_{2,3}=\diag\left(2,2\right)$.
\begin{figure}[h]
\begin{picture} (180,200)(250,70)
 
\put(300,100){\vector(0,1){100}}
\put(300,100){\vector(1,0){100}}

\dashline{10}(300,100)(400,200)

\put(330,130){\circle{10}}
\put(300,130){\circle*{10}}
\put(275,120){\makebox(0,0){$B_{1,0}$}}
\put(360,130){\circle{10}}
\put(300,130){\line(1,0){60}}

\put(360,100){\circle*{10}}
\put(375,90){\makebox(0,0){$B_{2,2}$}}
\put(360,160){\circle*{10}}
\put(380,155){\makebox(0,0){$B_{2,3}$}}
\put(380,125){\makebox(0,0){$B_{1,1}$}}
\put(360,100){\line(0,1){60}}
\end{picture}
\caption{\label{n=2 second split}Second split, $n=m=2$: $B_{2,0}=B_{2,1}=B_{1,0}=\diag\left(0,1\right)$, $B_{2,2}=\diag\left(2,0\right)$, $B_{2,3}=\diag\left(2,2\right)$.}
\end{figure}
After the second split, we obtain
\[\nu_{2}=\nu_{I}=\sum_{j=0}^{3}\lambda_{2,j}\delta_{B_{2,j}}=\frac{1}{2}\delta_{\diag\left(0,1\right)}+\frac{1}{4}\delta_{\diag \left(2,0\right)}+\frac{1}{4}\delta_{\diag\left(2,2\right)}.\]

In the case $n=3,m=2$, $E$ is the set of matrices of rank less than $2$. In order to exemplify the passage from step $k-1$ to step $k$, if we start with a matrix in $S^{k-1}_{1}$, the construction is the same as in the case $n=m=2$, whereas if we start with $A\in S_{0}^{k-1}$, we have $A=(k-1)I$ and Figures \ref{n=3 first split}, \ref{n=3 second split} and \ref{n=3 third split} show the construction of  $\nu_{A}$. 
\begin{figure}[h]
\begin{picture}(180,180)(-150,-80)
\put(0,0){\vector(0,1){100}}
\put(0,0){\vector(1,0){100}}
\put(0,0){\vector(-1,-1){70}}

\put(50,40){\circle{5}}\put(55,48){\makebox(0,0){$B_{0,0}$}}

\put(-60,40){\line(1,0){140}}
\put(-60,40){\circle*{5}}\put(-50,30){\makebox(0,0){$B_{1,0}$}}
\put(80,40){\circle*{5}}\put(90,30){\makebox(0,0){$B_{1,1}$}}
\end{picture}
\caption{\label{n=3 first split}First split, $n=3$, $m=2$: $B_{0,0}=(k-1)I$, $B_{1,0}=\diag\left(k-1,0,k-1\right)$, $B_{1,1}=\diag\left(k-1,k,k-1\right)$.}
\end{figure}
\begin{figure}[h]
\begin{picture}(180,180)(-150,-100)
\put(0,0){\vector(0,1){150}}
\put(0,0){\vector(1,0){180}}
\put(0,0){\vector(-1,-1){100}}


\put(50,40){\circle{5}}

\put(-60,40){\line(1,0){140}}
\put(-60,40){\circle{5}}\put(-85,45){\makebox(0,0)[l]{$B_{1,0}$}}
\put(80,40){\circle{5}}\put(80,30){\makebox(0,0)[l]{$B_{1,1}$}}

\put(0,100){\line(-1,-1){80}}
\put(0,100){\circle*{5}}\put(5,105){\makebox(0,0)[l]{$B_{2,0}$}}
\put(-80,20){\circle*{5}}\put(-90,10){\makebox(0,0)[l]{$B_{2,1}$}}

\put(140,100){\line(-1,-1){80}}
\put(140,100){\circle*{5}}\put(140,105){\makebox(0,0)[l]{$B_{2,2}$}}
\put(60,20){\circle*{5}}\put(60,10){\makebox(0,0)[l]{$B_{2,3}$}}
\end{picture}
\caption{\label{n=3 second split}Second split, $n=3$, $m=2$: $B_{2,0}=\diag\left(0,0,k-1\right)$, $B_{2,1}=\diag\left(k,0,k-1\right)$, $B_{2,2}=\diag\left(0,k,k-1\right)$, $B_{2,3}=\diag\left(k,k,k-1\right)$.}
\end{figure}
\begin{figure}[h]
\begin{picture}(180,200)(-150,-100)
\put(0,0){\vector(0,1){150}}
\put(0,0){\vector(1,0){180}}
\put(0,0){\vector(-1,-1){100}}


\put(50,40){\circle{5}}

\put(-60,40){\line(1,0){140}}
\put(-60,40){\circle{5}}
\put(80,40){\circle{5}}

\put(0,100){\line(-1,-1){80}}
\put(0,100){\circle*{5}}\put(5,105){\makebox(0,0)[l]{$B_{2,0}$}}
\put(-80,20){\circle{5}}\put(-100,15){\makebox(0,0)[l]{$B_{2,1}$}}

\put(140,100){\line(-1,-1){80}}
\put(140,100){\circle{5}}\put(145,95){\makebox(0,0)[l]{$B_{2,2}$}}
\put(60,20){\circle{5}}\put(65,10){\makebox(0,0)[l]{$B_{2,3}$}}

\put(-80,-80){\line(0,1){130}}
\put(-80,-80){\circle*{5}}\put(-75,-85){\makebox(0,0)[l]{$B_{3,2}$}}
\put(-80,50){\circle*{5}}\put(-95,60){\makebox(0,0)[l]{$B_{3,3}$}}

\put(140,0){\line(0,1){130}}
\put(140,0){\circle*{5}}\put(145,-10){\makebox(0,0)[l]{$B_{3,4}$}}
\put(140,130){\circle*{5}}\put(145,125){\makebox(0,0)[l]{$B_{3,5}$}}

\put(60,-80){\line(0,1){130}}
\put(60,-80){\circle*{5}}\put(65,-85){\makebox(0,0)[l]{$B_{3,6}$}}
\put(60,50){\circle*{5}}\put(60,60){\makebox(0,0)[l]{$B_{3,7}$}}

\end{picture}
\caption{\label{n=3 third split}Third split, $n=3$, $m=2$: $B_{3,0}=B_{3,1}=B_{2,0}=\diag\left(0,0,k-1\right)$, $B_{3,2}=\diag\left(k,0,0\right)$, $B_{3,3}=\diag\left(k,0,k\right)$, $B_{3,4}=\diag\left(0,k,0\right)$, $B_{3,5}=\diag\left(0,k,k\right)$, $B_{3,6}=\diag\left(k,k,0\right)$, $B_{3,7}=kI$.}
\end{figure}

We get at the end
\begin{align*}
\nu_{A}&=\sum_{j=0}^{7}\lambda_{3,j}\delta_{B_{3,j}}=\frac{1}{k^{2}}\delta_{\diag\left(0,0,k-1\right)}+\frac{(k-1)^{2}}{k^{3}}\delta_{\diag\left(k,0,k\right)}+\frac{k-1}{k^{3}}\delta_{\diag\left(k,0,0\right)}\\
&+\frac{(k-1)^{2}}{k^{3}}\delta_{\diag\left(0,k,k\right)}+\frac{k-1}{k^{3}}\delta_{\diag\left(0,k,0\right)}+\frac{(k-1)^{2}}{k^{3}}\delta_{\diag\left(k,k,0\right)}+\frac{(k-1)^{3}}{k^{3}}\delta_{kI}.
\end{align*}
The construction of $\nu_{k}$ would entail the analogous construction for each $A\in S_{0}^{k-1}\cup S_{1}^{k-1}$.

\section{General notation}\label{sect: notation}

We explain the general notation used throughout the paper, most of which is standard.

In the whole paper, $\Om$ is an open, non-empty bounded set of $\R^n$.

We denote by $\R^{n\times n}$ the set of $n\times n$ matrices, by $\Gamma_{+}$ its subset of symmetric positive semidefinite matrices, and by $SO(n) \subset \R^{n\times n}$ the orthogonal matrices with determinant $1$.

Given $A_i \in \R^{n\times n}$, the measure $\delta_{A_i}$ is the Dirac delta at $A_i$.
The barycenter of the probability measure $\nu=\sum_{i=1}^{N}\alpha_{i}\delta_{A_{i}}$ is
$\overline{\nu}=\sum_{i=1}^{N}\alpha_{i}A_{i}$. 

Given $A\in \R^{n\times n}$, let $\sigma_1(A)\leq\cdots\leq\sigma_n (A)$ denote its singular values. If the matrix $A$ is clear from the context, we will just indicate their singular values as $\sigma_1,\ldots,\sigma_n.$
In fact, in this paper we will always deal with $A \in \Gamma_+$, so its eigenvalues coincide with its singular values.
Its components are written $A_{\alpha, \beta}$ for $\alpha , \beta \in \{ 1, \ldots, n \}$.
Its operator norm is denoted by $|A|$, which coincides with $\sigma_n (A)$.
The norm of a $v \in \R^n$ is also denoted by $|v|$.

Given $a_1, \ldots, a_n \in \R$ the matrix $\diag (a_1, \ldots, a_n) \in \R^{n\times n}$ is the diagonal matrix with diagonal entries $a_1, \ldots, a_n$.

We will use the symbol $\lesssim$ when there exists a constant depending only on $n$ and $m$ such that the left hand side is less than or equal to the constant times the right hand side.

Given a set $E \subset \R^n$, we denote its characteristic function by $\chi_E$.
We write $\# E$ for the number of elements of $E$.
Its Lebesgue measure is denoted by $|E|$.

Given $a \in \R$, its integer part is denoted by $\floor{a}$.

Given $E \subset \R^n$, $\alpha \in (0,1]$ and a function $f : E \to \R^n$, we denote the H\"older seminorm, supremum norm and H\"older norm, respectively, as
\begin{align*}
 & \left| f \right|_{C^{\alpha} (E)} := \sup_{\substack{x_1, x_2 \in E \\ x_1 \neq x_2}} \frac{\left| f(x_2) - f(x_1)\right|}{\left| x_2 - x_1 \right|^{\alpha}} , \qquad \left\| f \right\|_{L^{\infty} (E)} = \sup_{x \in E} \left| f(x) \right| , \\
 & \left\| f \right\|_{C^{\alpha} (E)} :=  \left| f \right|_{C^{\alpha} (E)} + \left\| f \right\|_{L^{\infty} (E)} .
\end{align*}
We will write $f \in C^{\alpha} (E, \R^n)$ when $\left\| f \right\|_{C^{\alpha} (E)} < \infty$. 
Note that, if $f$ is continuous, the above norms and seminorms in $E$ coincide with those in $\overline{E}$.
In particular, we will identify $C^{\alpha} (E, \R^n)$ with $C^{\alpha} (\overline{E}, \R^n)$, the set of H\"older functions of exponent $\alpha$.
Of course, if $\alpha=1$, they are Lipschitz.
A homeomorphism $f$ is bi-H\"older if both $f$ and $f^{-1}$ are H\"older.

The identity function is denoted by $\id$.

We will say that a continuous mapping $f:\overline{\Omega}\to \R^n$ is \emph{piecewise affine} if there exists a countable family $\{\Omega_i\}_{i\in \N}$ of pairwise disjoint open subsets of $\Omega$ such that $f|_{\Omega_i}$ is affine for all $i\in\N$, and 
\[\left|\Omega\setminus \bigcup_{i\in \N}\Omega_i\right|=0.\]

\section{Approximation of laminates by functions}\label{sect: approximation of laminates by functions}

The following result, based on Astala, Faraco and Sz\'{e}kelyhidi \cite[Prop.\,2.3]{AsFaSz08}, shows how the gradient of a function can  approximate a laminate.
Given $A_1,A_2\in \Gamma_{+}$ we write $A_1\geq A_2$ when $A_1-A_2\in \Gamma_{+}$.
Note that, when $A \in \R^{n \times n}$, the same symbol $A$ is used to indicate a matrix (as in \emph{(\ref{gradient f})} below) and a linear function (as in \emph{(\ref{alpha f})}).

\begin{prop}\label{laminate-homeo}

Let $N\in \N$, $A_1,\ldots, A_N\in \Gamma_+$ and $L\geq 1$ be such that
\[A_i\geq L^{-1} I, \qquad |A_i|\leq L , \qquad i=1,\ldots, N.\]
Consider $\alpha_1, \ldots, \alpha_N \geq 0$ such that $\nu := \sum_{i=1}^{N}\alpha_{i}\delta_{A_{i}}$ is in $\mathcal{L}(\R^{n\times n})$ and call $A:=\overline{\nu}$.
Then, for every $\alpha\in (0,1)$, $0<\delta<\frac{1}{2}\min\{L^{-1},\min_{1\leq i<j\leq N}|A_{i}-A_{j}|\}$ and every bounded open set $\Omega\subset\R^{n}$, there exists a piecewise affine bi-Lipschitz homeomorphism $f:\Omega\to A\Om$ such that

\begin{enumerate}[(a)]
\item\label{ f is gradient} $f=\nabla u$ for some $u\in W^{2,\infty}(\Omega)$,
\item\label{frontera} $f(x)=Ax$ for $x\in\partial\Omega$,
\item\label{alpha f} $\|f-A\|_{C^{\alpha}(\overline{\Omega})}<\delta$,
\item\label{alpha f inv} $\|f^{-1}-A^{-1}\|_{C^{\alpha}(A\overline{\Omega})}<\delta$,
\item\label{gradient f} $|\{x\in\Omega: | D f(x)-A_{i}|<\delta \}|=\alpha_{i}|\Omega|$ for all $i=1,\ldots,N.$
\end{enumerate} 
\end{prop}
\begin{proof}

Parts \emph{(\ref{ f is gradient})}, \emph{(\ref{frontera})}, \emph{(\ref{alpha f})} and \emph{(\ref{gradient f})} are proved in \cite[Prop.\ 2.3]{AsFaSz08}. To prove \emph{(\ref{alpha f inv})} we first show that $f$ is bi-Lipschitz. We extend $f$ to an open ball $\Om'$ such that $\overline{\Om}\subset \Om'$ and $f(x)=Ax$ in $\Om'\backslash\Om$. Thus, $f$ is continuous in $\Om'$. We define, for each $\ve>0$,
\[\Om'_\ve=\{x\in\Omega': \dist(x,\partial \Omega')>\ve\}.\]
By \emph{(\ref{gradient f})} we get that $D f(x)\geq \frac{1}{2L} I$, and $|D f(x)|\leq 2L$ a.e.\ in $\Om'$. As $\Omega'$ is convex then $f$ is $2L$-Lipschitz. Let $\{\eta_{\ve}\}_{0<\ve\leq 1}$ be a standard family of mollifiers, and $f_\ve:=\eta_\ve *f\in C^{\infty}(\Om'_\ve)$ the mollification of $f$. Using that the matrices $M\in\R^{n\times n}_{\sym}$ satisfying $(2L)^{-1}I\leq M$ form a convex set, we find that there exists an $\ve_0>0$ such that if $\ve\le\ve_0$ then $D f_\ve(x)\geq \frac{1}{2L} I$ in $\Om'_\ve\supset\overline{\Omega}$.

For each $x,y\in\Omega'_{\ve}$, calling $h=y-x$, we have that
\begin{align*}
|f_\ve(y)-f_\ve(x)||h|&\geq |\langle f_\ve(y)-f_\ve(x),h\rangle|=\left|\left\langle\int_{0}^{1} D f_\ve(x+th) h\,dt,h\right\rangle\right|\\
&\geq\int_{0}^{1} \langle D f_\ve(x+th) h,h\rangle\,dt\geq \int_{0}^{1}\frac{1}{2L}|h|^2 \,dt= \frac{1}{2L}|h|^2,
\end{align*}
that is,
\[ |f_\ve(y)-f_\ve(x)|\geq\frac{1}{2L}|x-y|.\]
Using that $f_\ve\to f$ uniformly in $\overline{\Om}$ as $\ve\to 0$, we get
\[\frac{1}{2L}|x-y|\leq |f(y)-f(x)|.\]
 Hence $f$ is $K$-bi-Lipschitz in $\overline{\Om}$ with $K=2L$. Therefore, $f$ is a homeomorphism onto its image. The equalities $f(\Omega)=A\Omega$ and $f(\overline{\Om})=A\overline{\Omega}$ follows from standard results using the topological degree (e.g., \cite[Thms.\ 1 and 2]{Ball81}).
 
 Now we will estimate the $C^{\alpha}$ seminorm of $f^{-1}-A^{-1}$. For each $ y_1,y_2\in f(\overline{\Om})$, let $x_i=f^{-1}(y_i)$, $i=1,2$. Using that $f$ is $K$-bilipschitz and that $A\geq L^{-1}I$, we get
 \begin{align*}
 \frac{|f^{-1}(y_1)-A^{-1}y_1-f^{-1}(y_2)+A^{-1}y_2|}{|y_1-y_2|^{\alpha}}&\leq K^\alpha |A^{-1}| \frac{|Ax_1-f(x_1)-Ax_2+f(x_2)|}{|x_1-x_2|^{\alpha}}\\
 &\leq K^\alpha |A^{-1}| \left| f-A \right|_{C^{\alpha}(\overline{\Om})}\leq K^\alpha L\delta= 2 L^{\alpha+1}\delta,
 \end{align*}
so
 \begin{equation*}
 |f^{-1}-A^{-1}|_{C^{\alpha}(A\overline{\Omega})}\leq 2\delta L^{\alpha+1}.
 \end{equation*} 
 Therefore, $|f^{-1}-A^{-1}|_{C^{\alpha}(A\overline{\Omega})}$ can be done as small as we wish. Finally,
\[\|f^{-1}-A^{-1}\|_{L^{\infty} (A\overline{\Omega})}\leq |f^{-1}-A^{-1}|_{C^{\alpha}(A\overline{\Omega})}(\diam A\Omega)^{\alpha},\]
so $\|f^{-1}-A^{-1}\|_{C^{\alpha}(A\overline{\Omega})}$ is as small as we wish.
\end{proof}

\section{Cutting and pasting H\"older homeomorphisms}\label{se:cutting}

In this section we prove that if we modify a bi-H\"older homeomorphism in some sets by cutting and pasting other bi-H\"older homeomorphisms, the modified map is still a bi-H\"older homeomorphism.

First we show how to bound the $C^{\alpha}$ norm of a function in $\Omega$ with its $C^{\alpha}$ norms in a collection of subsets of $\Omega$ that covers $\Omega$ up to measure zero.
 \begin{lem}\label{Calpha norm}
 Let $\alpha\in (0,1]$, $\{\Omega_i\}_{i=1}^{\infty} \subset \Om$ pairwise disjoint open sets such that $|\Omega\setminus\bigcup_{i=1}^{\infty}\Omega_i|=0$, and $f:\overline{\Omega}\to\R^{n}$ such that $f(x)=0$ for all $x\in \bigcup_{i=1}^{\infty} \partial\Omega_i$.
Then
  \[\|f\|_{C^{\alpha}(\overline{\Omega})} \leq 2 \sup_{i \in \N} \|f\|_{C^{\alpha}(\overline{\Omega_i})} .\]
 \end{lem}
 \begin{proof}
We assume that the right hand side of the last inequality is finite.
Given $x,y\in\bigcup _{i=1}^{\infty}\overline{\Omega_i}$, let $i_0, i_1\in\N$ be such that $x\in \overline{\Omega_{i_0}}$ and $y\in \overline{\Omega_{i_1}}$. 

If $i_0=i_1$, then $|f(x)-f(y)|\leq \left| f \right|_{C^{\alpha} (\overline{\Omega_{i_0}})} |x-y|^{\alpha}$. Whereas if $i_0\neq i_1$, we have $\Omega_{i_0}\cap \Omega_{i_1}=\emptyset$. Consider 
\[\lambda_0:=\min\{\lambda\in [0,1]:x+\lambda (y-x)\in \partial \Omega_{i_0}\},
 \quad \lambda_1:=\max\{\lambda\in [0,1]:x+\lambda (y-x)\in \partial \Omega_{i_1}\}\]
and denote $x'=x+\lambda_0(y-x)$ and $y'=x+\lambda_1(y-x)$. Then, $f(x')=f(y')=0$, and, therefore, 
\begin{align*}
 |f(x)-f(y)| & \leq |f(x)-f(x')|+|f(y')-f(y)|\leq \left| f \right|_{C^{\alpha}(\overline{\Omega_{i_0}})} |x-x'|^{\alpha}+ \left| f \right|_{C^{\alpha}(\overline{\Omega_{i_1}})} |y'-y|^{\alpha}
 \\
 & \leq \left( \left| f \right|_{C^{\alpha}(\overline{\Omega_{i_0}})} + \left| f \right|_{C^{\alpha}(\overline{\Omega_{i_1}})} \right) |x-y|^{\alpha}.
\end{align*}
On the other hand, $\|f\|_{L^{\infty}\left(\bigcup _{i=1}^{\infty}\overline{\Omega_i}\right)}=\sup_{i\in\N} \|f\|_{L^{\infty}\left(\overline{\Omega_i}\right)}.$ Hence,
\[\|f\|_{C^{\alpha}(\bigcup_{i=1}^{\infty}\overline{\Omega_i})} \leq 2 \sup_{i \in \N} \|f\|_{C^{\alpha}(\overline{\Omega_i})} .\]
As $\bigcup_{i=1}^{\infty}\overline{\Omega_i}$ is dense in $\overline{\Omega}$, the required bound holds due to the uniform continuity.
\end{proof}

The main result of this section is the following.

\begin{lem}\label{glue homeomorphisms}
Let $f:\overline{\Om}\to \R^n$ be a homeomorphism such that $f$ and $f^{-1}$ are $C^{\alpha}$ for some $\alpha \in (0,1]$.
Let $\{ \omega_{i} \}_{i \in \N}\subset\Om$ be a family of pairwise disjoint open sets, and for each $i\in \N$ let $g_i:\overline{\omega_i}\to f(\overline{\omega_i})$ be a homeomorphism such that $g_i=f$ on $\partial \omega_i$,
 \[\sup_{i\in\N}\|f-g_i\|_{C^{\alpha}(\overline{\omega_i})} <\infty  \quad \text{and} \quad \sup_{i\in\N}\|f^{-1}-g_i^{-1}\|_{C^{\alpha}(f(\overline{\omega_i}))} <\infty .\]
Then, the function
\[\tilde{f} (x) :=\begin{cases}
f(x) &\text{if } x \in \overline{\Om}\setminus\bigcup_{i\in\N}\omega_i,\\
g_i (x) &\text{if } x \in \omega_i \text{ for some } i \in \N
\end{cases}\]
 is a homeomorphism between $\overline{\Om}$ and $f(\overline{\Om})$ such that $\tilde{f}$ and $\tilde{f}^{-1}$ are $C^{\alpha}$ and
\[
  \| f- \tilde{f}\|_{C^{\alpha} (\overline{\Omega})} \leq 2 \sup_{i\in\N}\|f-g_i\|_{C^{\alpha}(\overline{\omega_i})} , \qquad
 \| f^{-1} - \tilde{f}^{-1}\|_{C^{\alpha} (\overline{\Omega})} \leq 2 \sup_{i\in\N}\|f^{-1}-g_i^{-1}\|_{C^{\alpha}(f(\overline{\omega_i}))} .
\]
\end{lem}
\begin{proof}
Using that $f$ and $g_i$ are homeomorphisms we have that $f(\omega_i)=g_i(\omega_i)$, for each $i \in \N$.
Thus, it is clear that the function
\[ f(\overline{\Om}) \ni y \longmapsto \begin{cases}
 f^{-1} (y) &\text{if } y \in f\left(\overline{\Om}\setminus\bigcup_{i\in\N}\omega_i\right),\\
g_i^{-1} (y)&\text{if } y \in f(\omega_i) \text{ for some } i \in \N
\end{cases}\]
is the inverse of $\tilde{f}$.

Using Lemma \ref{Calpha norm}, we obtain that
\begin{align*}
 & \|\tilde{f}-f\|_{C^{\alpha}(\overline{\bigcup_{i\in\N}\omega_i})} \leq 2 \sup_{i\in\N}\|f-g_i\|_{C^{\alpha}(\overline{\omega_i})} , \\
 & \|\tilde{f}^{-1}-f^{-1}\|_{C^{\alpha}(f(\overline{\bigcup_{i\in\N}\omega_i}))} \leq 2 \sup_{i\in\N}\|f^{-1}-g_i^{-1}\|_{C^{\alpha}(f(\overline{\omega_i}))} .
\end{align*}
Moreover, when we call $F:= \tilde{f}-f$, we have that $F=0$ in $\overline{\Om}\setminus\bigcup_{i\in\N}\omega_i$.
In order to show that $F$ is $C^{\alpha}$ in $\overline{\Om}$, given $x_1 \in \overline{\Om}\setminus\bigcup_{i\in\N}\omega_i$ and $x_2 \in \bigcup_{i\in\N}\omega_i$, we take $x_3 = x_1 + \lambda (x_2 -x_1)$ for some $\lambda \in [0,1]$ such that $x_3 \in \partial \bigcup_{i\in\N}\omega_i$.
Then $F(x_1) = F(x_3) = 0$ and, hence,
\[
 \left| F(x_2) - F(x_1) \right| = \left| F(x_3) - F(x_1) \right| \leq \left\| F \right\|_{C^{\alpha} (\overline{\bigcup_{i\in\N}\omega_i})} \left| x_3 - x_1 \right|^{\alpha} \leq \left\| F \right\|_{C^{\alpha} (\overline{\bigcup_{i\in\N}\omega_i})} \left| x_2 - x_1 \right|^{\alpha} .
\]
This shows that $F \in C^{\alpha} (\overline{\Om}, \R^n)$.
Analogously, $\tilde{f}^{-1}-f^{-1}$ is $C^{\alpha}$ in $f (\overline{\Om})$ and the last bound of the statement also holds.
In particular, $\tilde{f}$ and $\tilde{f}^{-1}$ are $C^{\alpha}$, and, hence, $\tilde{f}$ is a homeomorphism between $\overline{\Omega}$ and $f(\overline{\Omega})$.
 \end{proof}

\section{Construction of the laminate and its approximation}\label{sect: lemmas}

This section constructs the sequence of laminates together with their approximations by functions.
We will continuously use the following sets and constants.

 For $j\in\N$, we define the sequence of open sets $E_j$ by
\[ E_j=\lbrace A\in \Gamma_+:  |A|>\frac{1}{2}+2^{-j} , \ \ 2^{-j-m}<\sigma_{i}(A) \max\lbrace |A|^{m-1},1\rbrace <2^{-j} \ \text{ for }1\leq i\leq n-m+1\rbrace. \]
For $a\in \N$, $0\leq a\leq n-m$, $j\in\N$ and $\mathcal{R}>\frac{1}{2}+2^{-j}$ such that 
\[\rho_{j,\mathcal{R}}:=\frac{3\cdot 2^{-j-2}}{\max\lbrace \mathcal{R}^{m-1},1\rbrace} < \mathcal{R},\]
 we define the closed sets 
\begin{align*}
E_{j,\mathcal{R}}^a =\lbrace &A\in \Gamma_+: |A|>\frac{1}{2}+2^{-j}, \ \ \sigma_{i}(A)=\rho_{j,\mathcal{R}} \text{ for } 1\leq i\leq a, \ \ \sigma_{i}\left( A \right)=\mathcal{R} \text{ for } a+1\leq i\leq n \rbrace.
\end{align*}
We also denote
\[ E_{j}^a=\bigcup_{\mathcal{R}\in (\frac{1}{2}+2^{-j},\infty)} E_{j,\mathcal{R}}^a . \]
The sets $E_j$ approximate the set of positive semidefinite matrices with rank less than $m$, and the sets $E_{j,\mathcal{R}}^{a}$ approximate the set 
\[\{ \mathcal{R} \, Q \, I_{a} \, Q^{T}: Q\in SO(n) \},\qquad I_{a}=\diag(\underbrace{0,\ldots,0}_{a},\underbrace{1\ldots,1}_{n-a}).\]
The number $\mathcal{R}$ plays the role of $k$ in Section \ref{sect: sketch of the proof} and eventually will tend to infinity.
The number $\rho_{j,\mathcal{R}}$ will tend to zero: the reason that it appears in the definition of $E_{j,\mathcal{R}}^a$ is that, even though $E_{j,\mathcal{R}}^a$ approximates a subset of matrices of rank $n-a$, we need them to be invertible.

Given $j\in\N$ and $\mathcal{R}>\frac{1}{2}+2^{-j}$ we define
\begin{equation}\label{rjR}
\begin{split}
r_{j,\mathcal{R}}:=\frac{1}{2}\min \Big\{ & 1-\left(\frac{2}{3}\right)^{\frac{1}{m-1}},\: \rho_{j,\mathcal{R}}\left(1-\frac{\max\{1,\mathcal{R}^{m-1}\}}{(\mathcal{R}+1)^{m-1}}\right),\: \frac{\rho_{j,\mathcal{R}}}{3} ,\: \mathcal{R}-\frac{1}{2}-2^{-j},  \\
 &  \frac{1}{2^{n}}\left(\mathcal{R}+1\right)^{-m-1} , \frac{1-\rho_{j,\mathcal{R}+1}}{\max \{1, \mathcal{R}\}} \Big\}
\end{split}
\end{equation}
and for $j\in \N$, $a_{0}, a\in\{0,\ldots,n-m\}$, $\mathcal{R} > \rho_{j,\mathcal{R}}$ we denote
\begin{equation}\label{C(j,R,a0,a)}
\begin{split}
C(j,\mathcal{R},a_{0},a)=\sum_{b=\max\lbrace 0,a_{0}+a-n \rbrace}^{\min\lbrace a_{0},a \rbrace}\binom{a_{0}}{b}\binom{n-a_{0}}{a-b}\left( \frac{\mathcal{R}+r_{j,\mathcal{R}}}{\mathcal{R}+1} \right)^{n-a_{0}-a+b} \\
\left( \frac{1}{\max\lbrace 1,\mathcal{R}\rbrace} \right)^{a-b}\left( \frac{2^{-j}}{\left(\mathcal{R}+1\right)\max\lbrace 1, \mathcal{R}^{m-1}\rbrace}\right)^{a_{0}-b}.
\end{split}
\end{equation}
We have chosen $r_{j,\mathcal{R}}$ small enough, depending on $\mathcal{R}$ (and, hence, on $\rho_{j,\mathcal{R}}$) to perform all the computations in this section: it will play the role of the $\delta$ of Proposition \ref{laminate-homeo}.

The next lemma constructs a laminate with the required integrability. The second part of its proof follows that of Section \ref{sect: sketch of the proof}.
\begin{lem}\label{lamlem}
Let $j\in \N$, $a_{0}\in\{0,\ldots,n-m\}$, $\mathcal{R}>0$ with $\rho_{j,\mathcal{R}}< \mathcal{R}$ and $A\in \Gamma_{+}$ be such that $\dist(A, E_{j,\mathcal{R}}^{a_{0}})<r_{j,\mathcal{R}}$. Then there exists $\nu\in \mathcal{L}(\R^{n\times n})$ such that $\overline{\nu}=A$,
 \[\supp \nu\subset \left(\bigcup_{a=0}^{n-m}E_{j,\mathcal{R}+1}^{a}\cup E_j \right)\cap \lbrace \xi\in\R^{n\times n}: |\xi|\leq \mathcal{R}+1 \rbrace\]
and for $0\leq a\leq n-m$,
\[\nu\left(E_{j,\mathcal{R}+1}^{a}\right)\leq C(j,\mathcal{R},a_{0},a).\]
\end{lem}
\begin{proof}
There exist $Q\in SO(n)$ and $B \in E_{j,\mathcal{R}}^{a_{0}}$ such that $A=Q\diag \left(\sigma_{1},\ldots,\sigma_{n}\right) Q^{T}$, with $0<\sigma_{1}\leq\cdots\leq \sigma_{n}$ and $|A-B| < r_{j,\mathcal{R}}$.
Using the inequality
\begin{equation}\label{eq:singvalues}
 \left| \sigma_i (A) - \sigma_i (B) \right| \leq |A - B| , \qquad i = 1, \ldots, n ,
\end{equation}
(see, e.g., \cite[Cor.\ 4.5]{GoGoKa90}), we find that
\begin{equation}\label{eq:sigma-r}
 |\sigma_{i}-\rho_{j,\mathcal{R}}|<r_{j,\mathcal{R}} \quad \text{for } 1\leq i\leq a_{0}, \quad \text{and} \quad |\sigma_{i}-\mathcal{R}|<r_{j,\mathcal{R}} \quad \text{for } a_{0}+1\leq i\leq n.
\end{equation}

In order to construct the desired laminate we prove that:
\begin{enumerate}[1)]
\item\label{lamlem sigma less R+1} $\sigma_{n}<\mathcal{R}+1.$
\item\label{lamlem epsilon less sigma}$\rho_{j,\mathcal{R}+1}<\sigma_{1}.$
\item\label{lamlem epsilon por sigma in Ej} $2^{-j-m}<\rho_{j,\mathcal{R}+1} \max\left\lbrace 1,\sigma_{n}^{m-1}\right\rbrace < 2^{-j}.$
\end{enumerate}
Inequality \ref{lamlem sigma less R+1}) is obvious thanks to \eqref{eq:sigma-r} since $r_{j,\mathcal{R}}<1$. By \eqref{eq:sigma-r}, the definition of $\rho_{j,\mathcal{R}+1}$ and \eqref{rjR} we obtain 
\[\rho_{j,\mathcal{R}+1}=\frac{3\cdot 2^{-j-2}}{\left(\mathcal{R}+1\right)^{m-1}}=\frac{\rho_{j,\mathcal{R}}\max\lbrace 1, \mathcal{R}^{m-1}\rbrace}{\left(\mathcal{R}+1\right)^{m-1}}<\rho_{j,\mathcal{R}}-r_{j,\mathcal{R}}<\sigma_{1},\]
where in the last inequality we have differentiated the cases $a_0=0$ and $a_0>0$.
So we have \ref{lamlem epsilon less sigma}).
Lastly, we prove \ref{lamlem epsilon por sigma in Ej}). On the one hand,
\[\rho_{j,\mathcal{R}+1} \max\left\lbrace 1,\sigma_{n}^{m-1}\right\rbrace\leq \rho_{j,\mathcal{R}+1} \max\left\lbrace 1,(\mathcal{R}+r_{j,\mathcal{R}})^{m-1}\right\rbrace\leq \rho_{j,\mathcal{R}+1} (\mathcal{R}+1)^{m-1}=3\cdot 2^{-j-2}<2^{-j}\]
and, on the other hand, 
\[\rho_{j,\mathcal{R}+1} \max\left\lbrace 1,\sigma_{n}^{m-1}\right\rbrace\geq \rho_{j,\mathcal{R}+1} \max\left\lbrace 1,(\mathcal{R}-r_{j,\mathcal{R}})^{m-1}\right\rbrace=\frac{3\cdot 2^{-j-2}\max\left\lbrace 1,(\mathcal{R}-r_{j,\mathcal{R}})^{m-1}\right\rbrace}{(\mathcal{R}+1)^{m-1}}.\]
Therefore, if $\mathcal{R}\leq 1$ we have 
\[\frac{3\cdot 2^{-j-2}\max\left\lbrace 1,(\mathcal{R}-r_{j,\mathcal{R}})^{m-1}\right\rbrace}{(\mathcal{R}+1)^{m-1}}\geq 3\cdot 2^{-j-m-1}>2^{-j-m},\]
whereas if $\mathcal{R}> 1$ we use (\ref{rjR}) to obtain
\[3\cdot 2^{-j-2} \frac{\max\left\lbrace 1,(\mathcal{R}-r_{j,\mathcal{R}})^{m-1}\right\rbrace}{(\mathcal{R}+1)^{m-1}}\geq 3\cdot 2^{-j-2}\frac{(\mathcal{R}-r_{j,\mathcal{R}})^{m-1}}{(\mathcal{R}+1)^{m-1}}\geq 3\cdot 2^{-j-m-1} (1-r_{j,\mathcal{R}})^{m-1}>2^{-j-m}.\]
Thus, \ref{lamlem epsilon por sigma in Ej}) is proved.

Now we build the laminate, following the lines of Section \ref{sect: sketch of the proof}. We shall construct families
\begin{equation}\label{eq:Bl}
 \{B_{\ell,i}\}_{\substack{\ell=0,\ldots,n\\ i=0,\ldots 2^{\ell}-1}}\subset \Gamma_{+} \quad \text{and} \quad \{\lambda_{\ell,i}\}_{\substack{\ell=0,\ldots,n\\ i=0,\ldots 2^{\ell}-1}}\subset [0,1]
\end{equation}
by finite induction on $\ell$.
Let $B_{0,0}=A$, $\lambda_{0,0}=1$ and for $0\leq \ell\leq   n-1$,  $0\leq i\leq 2^{\ell}-1$, we assume $\{B_{\ell,i}\}_{i=0}^{2^{\ell}-1}$ and $\{\lambda_{\ell,i}\}_{i=0}^{2^\ell-1}$ have been defined, $Q^{T}B_{\ell,j}Q$ is diagonal, $\lambda_{\ell,i}\geq 0$,
\begin{equation}\label{eq:inductionBl}
 \sum_{i=0}^{2^{\ell}-1}\lambda_{\ell,i}=1,\qquad B_{0,0}=\sum_{i=0}^{2^{\ell}-1}\lambda_{\ell,i} B_{\ell,i},\qquad \sum_{i=0}^{2^{\ell}-1}\lambda_{\ell,i}\delta_{B_{\ell,i}}\in\mathcal{L}(\R^{n\times n})
\end{equation}
and
\begin{equation}\label{lamlem autovalores de Bli > l}
\left(Q^{T}B_{\ell,i}Q\right)_{\alpha,\alpha}=\sigma_{\alpha} \quad\text{if } \alpha=\ell+1,\ldots, n.
\end{equation}
We also assume that if $B_{\ell,i}\notin E_j$ then
\begin{equation}\label{lamlem autovalores de Bli leq l}
\left(Q^{T}B_{\ell,i}Q\right)_{\alpha,\alpha}\in \{\rho_{j,\mathcal{R}+1},\mathcal{R}+1\}, \qquad \alpha=1,\ldots, \ell,
\end{equation}
 and when we let 
 \[\beta_{\ell,i}:=\#\{\alpha\in\{1,\ldots,\min\{a_{0},\ell\}\}: \left(Q^T B_{\ell,i}Q\right)_{\alpha,\alpha}=\rho_{j,\mathcal{R}+1}\},\]
\[\gamma_{\ell,i}:=\#\{\alpha\in\{a_{0}+1,\ldots,\ell\}: \left(Q^T B_{\ell,i}Q\right)_{\alpha,\alpha}=\rho_{j,\mathcal{R}+1}\},\]
then
\begin{equation}\label{eq:b+g}
 \beta_{\ell,i}+\gamma_{\ell,i}\leq n-m,
\end{equation}
and, calling
\begin{equation}\label{eq:UVW}
 U:= \frac{2^{-j}}{(\mathcal{R}+1)\max\{1,\mathcal{R}^{m-1}\}} , \qquad V:= \frac{1}{\max\{1,\mathcal{R}\}} , \qquad W:= \frac{\mathcal{R}+r_{j,\mathcal{R}}}{\mathcal{R}+1} ,
\end{equation}
we have
\begin{equation}\label{eq:boundl}
 \lambda_{\ell,i}\leq U^{\min\{a_{0},\ell\}-\beta_{\ell,i}} \, V^{\gamma_{\ell,i}} \, W^{\max\{0,\ell-a_{0}-\gamma_{\ell,i}\}}.
\end{equation}
We assume additionally that for each $B_{\ell,i'}\notin E_j$ such that $i'\neq i$, we have $B_{\ell,i'}\neq B_{\ell,i}$.
 
With the above induction hypotheses, we construct $\{B_{\ell+1,i}\}_{i=0}^{2^{\ell+1}-1}$ and $\{\lambda_{\ell+1,i}\}_{i=0}^{2^{\ell+1}-1}$ as follows. For any $0\leq i\leq 2^{\ell}-1$, define
\[ B_{\ell+1,2i}=\begin{cases}
B_{\ell,i}-Q\diag\left(\underbrace{0,\ldots,0}_{\ell}, \sigma_{\ell+1}-\rho_{j,\mathcal{R}+1},\underbrace{0,\ldots,0}_{n-\ell-1}  \right)Q^T, & \text{if }  B_{\ell,i}\notin E_{j},\\
B_{\ell,i}, & \text{if } B_{\ell,i}\in E_{j},
\end{cases}\]
\[ B_{\ell+1,2i+1}=\begin{cases}
B_{\ell,i}+Q\diag\left(\underbrace{0,\ldots,0}_{\ell},\mathcal{R}+1-\sigma_{\ell+1},\underbrace{0,\ldots,0}_{n-\ell-1}  \right)Q^T, & \text{if } B_{\ell,i}\notin E_{j},\\
B_{\ell,i}, & \text{if } B_{\ell,i}\in E_{j},
\end{cases}\]
\[ \lambda_{\ell+1,2i}=\lambda_{\ell,i}\frac{\mathcal{R}+1-\sigma_{\ell+1}}{\mathcal{R}+1-\rho_{j,\mathcal{R}+1}}\quad\text{and}\quad\lambda_{\ell+1,2i+1}=\lambda_{\ell,i}\frac{\sigma_{\ell+1}-\rho_{j,\mathcal{R}+1}}{\mathcal{R}+1-\rho_{j,\mathcal{R}+1}}.\]
So $\rank\left( B_{\ell+1,2i}-B_{\ell+1,2i+1} \right)\leq 1$, $\lambda_{\ell+1,2i}\geq 0$ by \ref{lamlem sigma less R+1}), $\lambda_{\ell+1,2i+1}\geq 0$ by \ref{lamlem epsilon less sigma}), and 
\[B_{\ell,i}=\frac{\mathcal{R}+1-\sigma_{\ell+1}}{\mathcal{R}+1-\rho_{j,\mathcal{R}+1}} B_{\ell+1,2i}+\frac{\sigma_{\ell+1}-\rho_{j,\mathcal{R}+1}}{\mathcal{R}+1-\rho_{j,\mathcal{R}+1}}B_{\ell+1,2i+1}. \]
With this, we can easily see that properties \eqref{eq:inductionBl} hold for $\ell+1$.
In what follows, $0\leq i\leq 2^{\ell+1}-1$.
We have
\[\left(Q^{T} B_{\ell+1,i}Q\right)_{\alpha,\alpha}=\left( Q^{T} B_{\ell,\floor{\frac{i}{2}}}Q\right)_{\alpha,\alpha},\qquad \alpha\neq \ell+1,\]
\[\left( Q^{T} B_{\ell+1,i} Q\right)_{\ell+1,\ell+1}=\begin{cases}
\rho_{j,\mathcal{R}+1}, & \text{if } B_{\ell,\floor{\frac{i}{2}}}\notin E_{j}, i \text{ even},\\
\mathcal{R}+1, & \text{if } B_{\ell,\floor{\frac{i}{2}}}\notin E_{j}, i \text{ odd},\\
\left( Q^{T}B_{\ell,\floor{\frac{i}{2}}}Q\right)_{\ell+1,\ell+1}, & \text{if } B_{\ell,\floor{\frac{i}{2}}}\in E_{j}.
\end{cases}\]
Therefore, property \eqref{lamlem autovalores de Bli > l} holds for $\ell+1$.
Now fix $\ell, i$ such that $B_{\ell+1,i}\notin E_{j}$.
Then $B_{\ell+1,\floor{\frac{i}{2}}}\notin E_{j}$, property \eqref{lamlem autovalores de Bli leq l} holds for $\ell+1$, and
 \[\beta_{\ell+1,i}=\begin{cases}
 \beta_{\ell,\frac{i}{2}}+1 & \text{if } i \text{ is even}, \, \ell<a_0,\\
 \beta_{\ell,\frac{i}{2}} & \text{if } i \text{ is even}, \, \ell\geq a_0,\\
 \beta_{\ell,\frac{i-1}{2}} & \text{if } i \text{ is odd},
 \end{cases}\qquad
 \gamma_{\ell+1,i}=\begin{cases}
 \gamma_{\ell,\frac{i}{2}} & \text{if } i \text{ is even}, \, \ell<a_0,\\
 \gamma_{\ell,\frac{i}{2}}+1 & \text{if } i \text{ is even}, \, \ell\geq a_0,\\
 \gamma_{\ell,\frac{i-1}{2}} & \text{if } i \text{ is odd}.
 \end{cases}\]
Using \eqref{eq:b+g} we find that $\beta_{\ell+1,i}+\gamma_{\ell+1,i}\leq n-m+1$.
On the other hand, we have shown that
\[
 \sigma_{\alpha} (B_{\ell+1,i}) \in \{ \rho_{j,\mathcal{R}+1}, \mathcal{R}+1 , \sigma_{\ell+2} , \ldots, \sigma_n \} , \qquad \alpha = 1, \ldots, n .
\]
Thus, if we had $\beta_{\ell+1,i}+\gamma_{\ell+1,i}=n-m+1$ then, by \ref{lamlem sigma less R+1}) and \ref{lamlem epsilon less sigma}) we would get
\[
 \sigma_{\alpha} (B_{\ell+1,i}) = \rho_{j,\mathcal{R}+1} , \qquad \alpha = 1, \ldots, n - m + 1 .
\]
and by \ref{lamlem epsilon por sigma in Ej}), $B_{\ell+1,i}\in E_j$, which is a contradiction.
Therefore, \eqref{eq:b+g} holds for $\ell+1$.

Now let $i'\neq i$ be such that $B_{\ell+1,i'}\notin E_j$.
If $\floor{\frac{i}{2}}\neq \floor{\frac{i'}{2}}$, then $B_{\ell,\floor{\frac{i'}{2}}}\neq B_{\ell,\floor{\frac{i}{2}}}$, and, hence, $B_{\ell+1,i'}\neq B_{\ell+1,i}$, 
 whereas if $\floor{\frac{i}{2}}= \floor{\frac{i'}{2}}$, then $(B_{\ell+1,i'})_{\ell+1,\ell+1}\neq (B_{\ell+1,j})_{\ell+1,\ell+1}$, and, hence,
$B_{\ell+1,i'}\neq B_{\ell+1,i}$.

Now we bound $\lambda_{\ell+1,i}$.
Recall the notation \eqref{eq:UVW} and the induction hypothesis \eqref{eq:boundl}.
 If $i$ is even and $\ell<a_0$, we have $\max\{0,\ell+1-a_{0}-\gamma_{\ell,i}\}=0$ and, therefore,
\begin{align*}
\lambda_{\ell+1,i}&=\lambda_{\ell,\frac{i}{2}}\frac{\mathcal{R}+1-\sigma_{\ell+1}}{\mathcal{R}+1-\rho_{j,\mathcal{R}+1}}\leq \lambda_{\ell,\frac{i}{2}} \leq U^{\min\{a_{0},\ell\}-\beta_{\ell,\frac{i}{2}}} \, V^{\gamma_{\ell,\frac{i}{2}}} \, W^{\max\{0,\ell-a_{0}-\gamma_{\ell,\frac{i}{2}}\}}\\
&=U^{\min\{a_{0},\ell+1\}-\beta_{\ell+1,i}} \, V^{\gamma_{\ell+1,i}} \, W^{\max\{0,\ell+1-a_{0}-\gamma_{\ell+1,i}\}}.
 \end{align*}
 If $i$ is even and $\ell\geq a_0$, using \eqref{eq:sigma-r} and \eqref{rjR}, we have
 \[\frac{\mathcal{R}+1-\sigma_{\ell+1}}{\mathcal{R}+1-\rho_{j,\mathcal{R}+1}}\leq \frac{1+r_{j,\mathcal{R}}}{\mathcal{R}+1-\rho_{j,\mathcal{R}+1}}\leq \frac{1}{\max \{ 1, \mathcal{R}\}} ,
\]
therefore
\begin{align*}
\lambda_{\ell+1,i}&=\lambda_{\ell,\frac{i}{2}}\frac{\mathcal{R}+1-\sigma_{\ell+1}}{\mathcal{R}+1-\rho_{j,\mathcal{R}+1}}\leq \lambda_{\ell,\frac{i}{2}}\frac{1}{\max\{1,\mathcal{R}\}}
\leq U^{\min\{a_{0},\ell\}-\beta_{\ell,\frac{i}{2}}} \, V^{\gamma_{\ell,\frac{i}{2}}+1} \, W^{\max\{0,\ell-a_{0}-\gamma_{\ell,\frac{i}{2}}\}}\\
&=U^{\min\{a_{0},\ell+1\}-\beta_{\ell+1,i}} \, V^{\gamma_{\ell+1,i}} \, W^{\max\{0,\ell+1-a_{0}-\gamma_{\ell+1,i}\}}.
\end{align*}
If $i$ is odd and $\ell<a_0$, then, by \eqref{eq:sigma-r} and the definition of $r_{j,\mathcal{R}}$ and  $\rho_{j,\mathcal{R}}$, we have
 \[\frac{\sigma_{\ell+1}-\rho_{j,\mathcal{R}+1}}{\mathcal{R}+1-\rho_{j,\mathcal{R}+1}}\leq \frac{\rho_{j,\mathcal{R}}+r_{j,\mathcal{R}}-\rho_{j,\mathcal{R}+1}}{\mathcal{R}+1-\rho_{j,\mathcal{R}+1}}\leq \frac{\rho_{j,\mathcal{R}}+r_{j,\mathcal{R}}}{\mathcal{R}+1}\leq \frac{4\rho_{j,\mathcal{R}}}{3(\mathcal{R}+1)}\leq\frac{2^{-j}}{\left(\mathcal{R}+1\right)\max\lbrace 1, \mathcal{R}^{m-1}\rbrace}\]
and
 \begin{align*}
\lambda_{\ell+1,i}&=\lambda_{\ell,\frac{i-1}{2}}\frac{\sigma_{\ell+1}-\rho_{j,\mathcal{R}+1}}{\mathcal{R}+1-\rho_{j,\mathcal{R}+1}}\leq \lambda_{\ell,\frac{i-1}{2}} U 
\leq U^{\min\{a_{0},\ell\}-\beta_{\ell,\frac{i-1}{2}}+1} \, V^{\gamma_{\ell,\frac{i-1}{2}}} \, W^{\max\{0,\ell-a_{0}-\gamma_{\ell,\frac{i-1}{2}}\}}\\
&=U^{\min\{a_{0},\ell+1\}-\beta_{\ell+1,i}} \, V^{\gamma_{\ell+1,i}} \, W^{\max\{0,\ell+1-a_{0}-\gamma_{\ell+1,i}\}}.
 \end{align*}
Finally, if $i$ is odd and $\ell\geq a_0$ we have $\gamma_{\ell,i}\leq \ell-a_{0}$ for all $i=0,\ldots, 2^{\ell}-1$, and 
 \[\frac{\sigma_{\ell+1}-\rho_{j,\mathcal{R}+1}}{\mathcal{R}+1-\rho_{j,\mathcal{R}+1}}\leq \frac{\mathcal{R}+r_{j,\mathcal{R}}-\rho_{j,\mathcal{R}+1}}{\mathcal{R}+1-\rho_{j,\mathcal{R}+1}}\leq \frac{\mathcal{R}+r_{j,\mathcal{R}}}{\mathcal{R}+1},\]
 so
 \begin{align*}
\lambda_{\ell+1,i}&=\lambda_{\ell,\frac{i-1}{2}}\frac{\sigma_{\ell+1}-\rho_{j,\mathcal{R}+1}}{\mathcal{R}+1-\rho_{j,\mathcal{R}+1}}\leq \lambda_{\ell,\frac{i-1}{2}} W
\leq U^{\min\{a_{0},\ell\}-\beta_{\ell,\frac{i-1}{2}}} \, V^{\gamma_{\ell,\frac{i-1}{2}}} \, W^{\max\{0,\ell-a_{0}-\gamma_{\ell,\frac{i-1}{2}}\}+1}\\
&=U^{\min\{a_{0},\ell+1\}-\beta_{\ell+1,i}} \, V^{\gamma_{\ell+1,i}} \, W^{\max\{0,\ell+1-a_{0}-\gamma_{\ell+1,i}\}}.
\end{align*}
With this, we finish the inductive construction of the families \eqref{eq:Bl}.
In particular, for all $0\leq i\leq 2^{n}-1$, if $B_{n,i}\notin E_j$ we have
\begin{equation}\label{eq:lni}
 \lambda_{n,i}\leq U^{a_{0}-\beta_{n,i}} \, V^{\gamma_{n,i}} \, W^{n-a_{0}-\gamma_{n,i}},
\end{equation}
 \[\left(Q^{T}B_{n,i}Q\right)_{\alpha,\alpha}\in \{\rho_{j,\mathcal{R}+1},\mathcal{R}+1\}, \qquad \alpha=1,\ldots, n,\]
and
\[a:=\#\{\alpha: (Q^T B_{n,i}Q)_{\alpha,\alpha}=\rho_{j,\mathcal{R}+1}\}=\beta_{n,i}+\gamma_{n,i}\leq n-m.\]
Therefore, by definition of $E^{a}_{j,\mathcal{R}+1}$, we get
$B_{n,i}\in E^{a}_{j,\mathcal{R}+1}$.
Hence, for all $0\leq i\leq 2^{n}-1$, we have proved that
\[B_{n,i}\in \bigcup_{a=0}^{n-m}E^{a}_{j,\mathcal{R}+1}\cup E_{j}.\]
We define 
\begin{equation*}
\nu=\sum_{i=0}^{2^{n}-1}\lambda_{n,i}\delta_{B_{n,i}},
\end{equation*}
which is a laminate by \eqref{eq:inductionBl}.

In order to estimate $\nu (E_{j,\mathcal{R}+1}^{a} )$, we observe that for $B_{n,i}\in E_{j,\mathcal{R}+1}^{a}$ we have $\max\lbrace 0,a_{0}+a-n \rbrace\leq \beta_{n,i}\leq \min\lbrace a_{0},a \rbrace$. Therefore
\begin{align*}
\nu\left(E_{j,\mathcal{R}+1}^{a}\right)&=\sum_{\substack{i: \beta_{n,i}+\gamma_{n,i}=a \\ B_{n,i}\in E_{j,\mathcal{R}+1}^{a}}}\lambda_{n,i}=\sum_{b=\max\lbrace 0,a_{0}+a-n \rbrace}^{\min\lbrace a_{0},a \rbrace} \: \sum_{\substack{i: \beta_{n,i}=b,\: \gamma_{n,i}=a-b \\ B_{n,i}\in E_{j,\mathcal{R}+1}^{a}}}\lambda_{n,i}\\
&\leq \sum_{b=\max\lbrace 0,a_{0}+a-n \rbrace}^{\min\lbrace a_{0},a \rbrace}\binom{a_{0}}{b}\binom{n-a_{0}}{a-b} U^{a_{0}-b} \, V^{a-b} \, W^{n-a_{0}-a+b},
\end{align*}
where we have used that the $B_{n,i}$ ($0 \leq i \leq 2^n -1$) in $E_{j,\mathcal{R}+1}^{a}$ are all different, as well as estimate \eqref{eq:lni}.
This concludes the proof.
\end{proof}

The following result constructs a function whose gradient approximates the laminate of the previous lemma and have the desired integrability. 
\begin{lem}\label{tau lemma}
Let $\alpha\in(0,1)$ and $\delta>0$. Then is a $j_1\in \N$ such that for any $j\geq j_{1}$, any bounded open set $\omega\subset \R^n$ and any $F\in \Gamma_{+}$ such that $\dist(F,\bigcup_{a=0}^{n-m}E_{j,|F|}^{a})<r_{j,|F|}$, there exists a piecewise affine homeomorphism $f\in W^{1,1}(\omega,F\omega)\cap C^{\alpha}(\overline{\omega},\overline{F\omega})$ such that
\begin{enumerate}[i)]
\item\label{tau border} $f(x)=Fx$ for all  $x\in\partial \omega,$
\item\label{tau C alpha} $\|f-F\|_{C^{\alpha}(\overline{\omega})}<\delta$ and $\|f^{-1}-F^{-1}\|_{C^{\alpha}(\overline{F\omega})}<\delta,$
\item\label{tau dist} $ D f(x)\in E_{j}$ a.e.\ $x\in \omega,$
\item\label{tau regularity} for all $t>0$,
\begin{equation*}
\frac{|\{x\in\omega:| D f(x)|>t\}|}{|\omega|}\lesssim |F|^m t^{-m}.
\end{equation*}

\end{enumerate}
\end{lem}
\begin{proof}
Let $\mathcal{R}=|F|$, $a_0\in\{0,\ldots,n-m\}$, $Q\in SO(n)$ and
\[A=Q\diag\,\left(\underbrace{\rho_{j,\mathcal{R}},\ldots,\rho_{j,\mathcal{R}},}_{a_{0}}\underbrace{\mathcal{R},\ldots,\mathcal{R}}_{n-a_{0}}\right)Q^{T}\in E_{j,\mathcal{R}}^{a_{0}}\]
be such that $| F-A |<r_{j,\mathcal{R}}$, and for $a=0,\ldots, n-m$, define the sets 
\[ S_{j,\mathcal{R}}^{a}:= \left\lbrace M\in \Gamma_{+}: \dist \left( M, E_{j,\mathcal{R}}^{a}\right) < r_{j, \mathcal{R}}\right\rbrace . \]
Note that the sets $S_{j,\mathcal{R}}^{0}, \ldots, S_{j,\mathcal{R}}^{n-m}$ are pairwise disjoint.
Indeed, if $S_{j,\mathcal{R}}^{a_1} \cap S_{j,\mathcal{R}}^{a_2} \neq \emptyset$ for some $a_1 \neq a_2$ we would obtain, thanks to inequality \eqref{eq:singvalues}, $|\mathcal{R} - \rho_{j,\mathcal{R}}| < 2 r_{j, \mathcal{R}}$, which contradicts the definition of $r_{j, \mathcal{R}}$.

Given $k\in \N$, we define $\tilde{k}=k+\mathcal{R}-1$.
We will construct by induction a sequence $\lbrace f_{k}\rbrace_{k \in \N}$ of piecewise affine homeomorphisms such that
\begin{enumerate}[\it (a)]
\item\label{tau f_k border} $f_k(x)=Fx $ for all  $x\in\partial \omega$.
\item\label{tau f_k C alpha} $\|f_k-f_{k-1}\|_{C^{\alpha}(\overline{\omega})}<2^{-k}\delta$ and $\|f_k^{-1}-f_{k-1}^{-1}\|_{C^{\alpha}(\overline{F\omega})}<2^{-k}\delta$.
\item\label{tau f_k dist kI} $ D f_{k}(x)\in E_{j}\cup\bigcup_{a=0}^{n-m}S_{j,\tilde{k}}^{a}$ for a.e.\ $x\in \omega$.
\item\label{tau bound of Dfk} $\left| Df_k \right| <\tilde{k}+1$ in $\omega\setminus \omega_k$, with $\omega_{k}:=\bigcup_{a=0}^{n-m}\omega_{k}^{a}$ and 
\[\omega_{k}^{a}:=\lbrace x\in\omega : f_{k} \text{ is affine in a neighbourhood of } x \text{ and } D f_{k}(x) \in S_{j,\tilde{k}}^{a} \rbrace,\quad 0 \leq a \leq n-m.\]
\item\label{tau bounds Om_k 1} There exists $j_1\in \N$ such that for any $j\geq j_{1}$ we have
\[
\frac{|\omega_{k}^{a}|}{|\omega|}\lesssim  \tilde{k}^{\frac{1}{2}+a-n}\quad\text{for } 0\leq a\leq n-m-1, \quad \text{and} \quad \frac{|\omega_{k}^{n-m}|}{|\omega|}\lesssim \tilde{k}^{-m} \sum_{d=1}^{k} \tilde{d}^{-\frac{4}{3}}.
\]

\item\label{omegak} $\omega_k \supset \omega_{k+1}$ and $f_{k+1} |_{\omega \setminus \omega_k} = f_k |_{\omega \setminus \omega_k}$.
\end{enumerate}

Note that the sets $\omega_{k}^{a}$ defined in \emph{(\ref{tau bound of Dfk})} are open and pairwise disjoint, since so are $S_{j,\mathcal{R}}^{a}$.
Recall also that $\tilde{d}$ stands for $d+ \mathcal{R} -1$.

For $k=0,1$ we see that the choices $f_0(x) = f_1(x) = Fx$, $\omega_{0}^{a_{0}} = \omega_{1}^{a_{0}}=\omega$ and $\omega_{0}^{a} = \omega_{1}^{a} = \emptyset$ for $a\neq a_0$ satisfy all the assumptions.

Fix $k\in \N$ and assume $f_k$ has been constructed. We obtain $f_{k+1}$ by modifying $f_{k}$ on the sets $\omega_{k}^{a}$. Since $f_{k}$ is piecewise affine, there exists a family $\{ \omega_i\}_{i \in \N} \subset \omega$ of pairwise disjoint open sets such that $|\omega \setminus \bigcup_{i \in \N} \omega_i| = 0$ and $f|_{\omega_i}$ is affine for each $i \in \N$.
More precisely, fix $a\in\{0,\ldots,n-m\}$ and define $\omega_{k,i}^{a} := \omega_i \cap \omega_k^a$ for each $i \in \N$, which is an open set.
From now on, we only deal with those $\omega_{k,i}^{a}$ that are non-empty.
Then there exist families $\{ A_{k,i}^{a} \}_{i \in \N}\subset S_{j,\tilde{k}}^{a}$ and $\{ b_{k,i}^{a} \}_{i \in \N} \subset \R^n$ such that $f_{k}(x)=A_{k,i}^{a} x+b_{k,i}^{a}$ for $x \in\omega_{k,i}^{a}$.

Let $\nu_{A_{k,i}^{a}}$ be the laminate of Lemma \ref{lamlem} that satisfies $\overline{\nu_{A_{k,i}^{a}}}=A_{k,i}^{a}$,
 \[\supp \nu_{A_{k,i}^{a}}\subset \left(\bigcup_{b=0}^{n-m}E_{j,\tilde{k}+1}^{b}\cup E_j \right)\cap \lbrace \xi\in\R^{n\times n}: |\xi|\leq \tilde{k}+1 \rbrace\]
and for $0\leq b\leq n-m$,
\[\nu_{A_{k,i}^{a}}\left(E_{j,\tilde{k}+1}^{b}\right)\leq C(j,\tilde{k},a,b).\]
We apply Proposition \ref{laminate-homeo} to that laminate and obtain a piecewise affine homeomorphism $g_{k,i}^{a}:\omega_{k,i}^{a}\to A_{k,i}^{a}\omega_{k,i}^{a}+b_{k,i}^{a}$ with
\begin{enumerate}[\it (a)]
\setcounter{enumi}{6}
\item\label{tau g_i border} $g_{k,i}^{a}(x)=A_{k,i}^{a} x+b_{k,i}^{a}$ on $\partial \omega_{k,i}^{a}$.
\item\label{tau nabla g_i} $| D g_{k,i}^{a}(x)|< \tilde{k}+2$ a.e.\ in $\omega_{k,i}^{a}$.
\item\label{tau g_i C alpha} $\|g_{k,i}^{a}-f_{k}\|_{C^{\alpha}(\overline{\omega_{k,i}^{a}})}< 2^{-k-2}\delta$ and $\|\big( g_{k,i}^{a} \big)^{-1}-f_{k}^{-1}\|_{C^{\alpha}(A_{k,i}^{a}\overline{\omega_{k,i}^{a}}+b_{k,i}^{a})} < 2^{-k-2}\delta$.
\item\label{tau dist g_i kI} $ D g_{k,i}^{a}(x)\in E_{j}\cup\bigcup_{b=0}^{n-m}S_{j,\tilde{k}+1}^{b}$ a.e.\ in $\omega_{k,i}^{a}$.
\item\label{tau bounds dist(g_i,(m+1)I)}
$\displaystyle \left|\lbrace x\in\omega_{k,i}^{a}: D g_{k,i}^{a}(x)\in S_{j,\tilde{k}+1}^{b}\rbrace \right| \leq C(j,\tilde{k},a,b) |\omega_{k,i}^{a}|.$
\end{enumerate}
In property \emph{(\ref{tau dist g_i kI})} we have used that $E_j$ is open.
We define the piecewise affine function
\[f_{k+1} (x) =\begin{cases}
f_{k} (x) & \text{if } x \in \overline{\omega} \setminus\bigcup_{i=1}^{\infty}\bigcup_{a=0}^{n-m}\omega_{k,i}^{a} , \\
g_{k,i}^{a} (x) & \text{if } x \in \omega_{k,i}^{a} \text{ for some } i \in \N \text{ and } a \in \{ 0, \ldots, n-m \},
\end{cases}\]
which is a homeomorphism due to Lemma \ref{glue homeomorphisms}.
Property \emph{(\ref{tau f_k border})} holds for $k+1$ since $f_{k+1} = f_k$ on $\partial \omega$.
Property \emph{(\ref{tau f_k C alpha})} holds for $k+1$ thanks to  \emph{(\ref{tau g_i C alpha})} and Lemma \ref{glue homeomorphisms}.
Property \emph{(\ref{omegak})} for $k+1$ follows easily from the construction.
Property \emph{(\ref{tau bound of Dfk})} for $k+1$ follows from \emph{(\ref{tau nabla g_i})} and \emph{(\ref{omegak})}.
Property \emph{(\ref{tau f_k dist kI})} for $k+1$ follows from \emph{(\ref{tau dist g_i kI})} and \emph{(\ref{omegak})}.
Finally, we have to prove \emph{(\ref{tau bounds Om_k 1})} for $k+1$. 

By definition of $f_{k+1}$, we have that, up to a set of measure zero,
\[
 \omega_{k+1}^{b}=\bigcup_{a=0}^{n-m}\bigcup_{i=1}^{\infty}\left\lbrace x\in\omega_{k,i}^{a} : D g_{k,i}^{a}(x)\in S_{j,\tilde{k}+1}^{b} \right\rbrace
\]
with disjoint union, hence for $b=0,\ldots,n-m$,
\begin{align*}
\frac{|\omega_{k+1}^{b}|}{|\omega|}&= \sum_{a=0}^{n-m}\sum_{i=1}^{\infty}\frac{|\omega_{k,i}^{a}|}{|\omega|}\frac{\left|\left\lbrace x\in\omega_{k,i}^{a} : D g_{k,i}^{a}(x)\in  S_{j,\tilde{k}+1}^{b} \right\rbrace\right|}{|\omega_{k,i}^{a}|}\leq \sum_{a=0}^{n-m}\frac{|\omega_{k}^{a}|}{|\omega|}C(j,\tilde{k},a,b) \\
&\lesssim\sum_{a=0}^{n-m-1}\tilde{k}^{\frac{1}{2}+a-n}C(j,\tilde{k},a,b)+ C(j,\tilde{k},n-m,b)\tilde{k}^{-m} \sum_{d=1}^{k} \tilde{d}^{-\frac{4}{3}},
\end{align*}
where we have used \emph{(\ref{tau bounds dist(g_i,(m+1)I)})} and \emph{(\ref{tau bounds Om_k 1})}.
So, in order to prove \emph{(\ref{tau bounds Om_k 1})} for $k+1$ it is enough to show that there exist $j_0\in\N$ and $k_{0}\in\N$ such that if $k\geq k_{0}$ and $j\geq j_0$ then
\begin{equation}\label{b<n-m,k+1}
\sum_{a=0}^{n-m-1}\tilde{k}^{\frac{1}{2}+a-n}C(j,\tilde{k},a,b)+ C(j,\tilde{k},n-m,b)\tilde{k}^{-m} \sum_{d=1}^{k} \tilde{d}^{-\frac{4}{3}}\leq \left(\tilde{k}+1\right)^{\frac{1}{2}+b-n}
\end{equation}
for $0\leq b\leq n-m-1$, and
\begin{equation}\label{b=n-m,k+1}
\sum_{a=0}^{n-m-1}\tilde{k}^{\frac{1}{2}+a-n}C(j,\tilde{k},a,n-m)+ C(j,\tilde{k},n-m,n-m)\tilde{k}^{-m} \sum_{d=1}^{k} \tilde{d}^{-\frac{4}{3}}\leq \left(\tilde{k}+1\right)^{-m} \sum_{d=1}^{k+1} \tilde{d}^{-\frac{4}{3}}.
\end{equation}
Recall from \eqref{C(j,R,a0,a)} that, for $k\geq 2$, 
\[
C(j,\tilde{k},a,b)= (\tilde{k}+1)^{-n+b} \! \sum_{\ell=\max\lbrace 0,a+b-n \rbrace}^{\min\lbrace a,b \rbrace}\binom{a}{\ell}\binom{n-a}{b-\ell} \, ( \tilde{k}+r_{j,\tilde{k}} )^{n-a-b+\ell} \; \tilde{k}^{m (\ell-a) + a - b} \; 2^{-j (a-\ell)} .
\]
Using the inequality $(\tilde{k}+r_{j,\tilde{k}})^{n-a-b+\ell} \leq \tilde{k}^{n-a-b+\ell} + 2^n \tilde{k}^{n-a-b+\ell-1} r_{j,\tilde{k}}$, we find that
\begin{multline*}
C(j,\tilde{k},a,b) \leq \\
 (\tilde{k}+1)^{b-n} \! \sum_{\ell=\max\lbrace 0,a+b-n \rbrace}^{\min\lbrace a,b \rbrace} \! \binom{a}{\ell}\binom{n-a}{b-\ell} \tilde{k}^{n + m(\ell-a) + \ell -2b} \left[1 + 2^{n} r_{j,\tilde{k}} \tilde{k}^{-1} \right] 2^{-j (a-\ell)} .
\end{multline*}
We will also use the cruder inequality
\[
 C(j,\tilde{k},a,b) \lesssim (\tilde{k}+1)^{b-n} \, \tilde{k}^{n + \min \{a,b\} (m+1) -a m  -2b} .
\]

In order to show inequalities \eqref{b<n-m,k+1} and \eqref{b=n-m,k+1}, we estimate $C(j,\tilde{k},a,b)$ according to whether $a$ or $b$ equal $n-m$ or are less than it.
We first observe that for $a,b\in\{0,\ldots,n-m\}$ and $\ell\in\{0,\ldots,\min\{a,b\}\}$  we have
\[
 \begin{cases}
 m(\ell-a)+\ell+a-2b =0 & \text{if } a=b=\ell, \\
m(\ell-a)+\ell+a-2b \leq -1 & \text{otherwise}.
\end{cases}
\]
In the case $0\leq b\leq n-m-1$ we have that exists a constant $c_1(n)$ depending on $n$ such that, for $k\geq 2$,
\begin{align*}
& \left(\tilde{k}+1\right)^{n-b}\sum_{a=0}^{n-m-1}\tilde{k}^{\frac{1}{2}+a-n}C(j,\tilde{k},a,b)\\
& \leq \sum_{a=0}^{n-m-1} \sum_{\ell=\max\lbrace 0,a+b-n \rbrace}^{\min\lbrace a,b \rbrace}\binom{a}{\ell} \binom{n-a}{b-\ell} \tilde{k}^{m(\ell-a)+\ell+a-2b} \left[\tilde{k}^{\frac{1}{2}}+2^{n}r_{j,\tilde{k}}\tilde{k}^{-\frac{1}{2}}\right] 2^{-j (a-\ell)} \\
& \leq \tilde{k}^{\frac{1}{2}}+c_1(n)\frac{2^{-j}}{\tilde{k}^{\frac{1}{2}}} ,
\end{align*}
so
\begin{equation}\label{a<n-m,b<n-m}
\left(\tilde{k}+1\right)^{n-\frac{1}{2}-b}\sum_{a=0}^{n-m-1}\tilde{k}^{\frac{1}{2}+a-n}C(j,\tilde{k},a,b) \leq \left( \frac{\tilde{k}}{\tilde{k}+1}\right)^{\frac{1}{2}}+c_1(n)\frac{2^{-j}}{\tilde{k}^{\frac{1}{2}}(\tilde{k}+1)^{\frac{1}{2}}} ,
\end{equation}
whereas
\[
C(j,\tilde{k},n-m,b)\tilde{k}^{-m} \left(\tilde{k}+1\right)^{n-b} \lesssim 
 \tilde{k}^{n + b (m+1) -(n-m) m  -2b-m} \leq \tilde{k}^{1-m} \leq \tilde{k}^{-1} 
\]
so
\begin{equation}\label{b<n-m,c(n)}
 C(j,\tilde{k},n-m,b)\tilde{k}^{-m} \sum_{d=1}^{k} \tilde{d}^{-\frac{4}{3}}\left(\tilde{k}+1\right)^{n-\frac{1}{2}-b} \leq c_1(n)\tilde{k}^{-\frac{3}{2}}.
\end{equation}
Given the previous constant $c_1(n)$, let $j_1\in\N$ be such that for all $j\geq j_1$ and $k \geq 2,$
\begin{equation}\label{c(n)j1}
 \left( \frac{\tilde{k}}{\tilde{k}+1}\right)^{\frac{1}{2}}+c_1(n)\left(\frac{2^{-j}}{\tilde{k}^{\frac{1}{2}}(\tilde{k}+1)^{\frac{1}{2}}}+\tilde{k}^{-\frac{3}{2}}\right)\leq 1.
\end{equation}
Then using (\ref{a<n-m,b<n-m}), (\ref{b<n-m,c(n)}) and (\ref{c(n)j1})
\begin{align*}
&\left(\tilde{k}+1\right)^{n-\frac{1}{2}-b}\left[\sum_{a=0}^{n-m-1}\tilde{k}^{\frac{1}{2}+a-n}C(j,\tilde{k},a,b)+ C(j,\tilde{k},n-m,b)\tilde{k}^{-m} \sum_{d=1}^{k} \tilde{d}^{-\frac{4}{3}}\right]\\
&\leq \left( \frac{\tilde{k}}{\tilde{k}+1}\right)^{\frac{1}{2}}+c_1(n)\left(\frac{2^{-j}}{\tilde{k}^{\frac{1}{2}}(\tilde{k}+1)^{\frac{1}{2}}}+\tilde{k}^{-\frac{3}{2}}\right)\leq 1 ,
\end{align*}
which proves \eqref{b<n-m,k+1}.
In the case $b=n-m$,
\[
 \left( \tilde{k}+1 \right)^{m} C(j,\tilde{k},a,n-m) \lesssim \tilde{k}^{2m + a -n} ,
\]
so
\[
 \sum_{a=0}^{n-m-1}\tilde{k}^{\frac{1}{2}+a-n}C(j,\tilde{k},a,n-m)\left( \tilde{k}+1 \right)^{m} \lesssim \sum_{a=0}^{n-m-1} \tilde{k}^{\frac{1}{2}+2 (m+a-n)} \lesssim \tilde{k}^{-\frac{3}{2}}
\]
and, hence, there exists a constant $c_2(n)$ such that
\begin{equation}\label{sum,c(n)}
 \sum_{a=0}^{n-m-1}\tilde{k}^{\frac{1}{2}+a-n}C(j,\tilde{k},a,n-m)\left( \tilde{k}+1 \right)^{m} \leq c_2(n)\tilde{k}^{-\frac{3}{2}}.
\end{equation}
Recall that
\[2^{n}r_{j,\tilde{k}}\leq \frac{1}{2}\left(\tilde{k}+1\right)^{-m-1}.\]
Then, splitting the following sum in the case $\ell=n-m$ and the case $\ell<n-m$, we have
\begin{equation}\label{sum,k+1}
\begin{split}
 & C(j,\tilde{k},n-m,n-m)\tilde{k}^{-m} (\tilde{k}+1)^{m} \\
 & \leq \sum_{\ell=\max\lbrace 0,n-2m \rbrace}^{n-m}\binom{n-m}{\ell}\binom{m}{n-m-\ell} \tilde{k}^{-n + m (\ell -n+m+1) + \ell} \left[ 1 + 2^{n} r_{j,\tilde{k}} \tilde{k}^{-1} \right] \\
 & \leq 1+ c_2 (n) \tilde{k}^{-m-1} .
\end{split}
\end{equation}
In addition, there exists $k_0\in\N$ such that for $k\geq k_{0}$
 \begin{equation}\label{c(n,k+1)}
 c_2(n) \left( \tilde{k}^{-\frac{3}{2}} +  \tilde{k}^{-m-1} \sum_{d=1}^{\infty} \tilde{d}^{-\frac{4}{3}} \right) \leq \left(\tilde{k}+1\right)^{-\frac{4}{3}}.
 \end{equation}
Therefore, for $j\in\N$ and $k\geq k_0$ we use \eqref{sum,c(n)}, \eqref{sum,k+1} and \eqref{c(n,k+1)} to get
\begin{align*}
&(\tilde{k}+1)^{m}\left[\sum_{a=0}^{n-m-1}\tilde{k}^{\frac{1}{2}+a-n}C(j,\tilde{k},a,n-m)+ C(j,\tilde{k},n-m,n-m)\tilde{k}^{-m} \sum_{d=1}^{k} \tilde{d}^{-\frac{4}{3}}\right]\\
&\leq c_2(n)\tilde{k}^{-\frac{3}{2}}+ \sum_{d=1}^{k} \tilde{d}^{-\frac{4}{3}}\left( 1+ c_2 (n) \tilde{k}^{-m-1} \right)\leq \sum_{d=1}^{k+1} \tilde{d}^{-\frac{4}{3}}.
\end{align*}
This proves (\ref{b=n-m,k+1}) and the construction of $\{f_k\}_{k \in \N}$ is finished.

From \emph{(\ref{tau bounds Om_k 1})} we obtain
\begin{equation}\label{tau size Om_k}
\frac{|\omega_{k}|}{|\omega|}= \sum_{a=0}^{n-m} \frac{|\omega_{k}^{a}|}{|\omega|}\lesssim\sum_{a=0}^{n-m-1} \tilde{k}^{\frac{1}{2}+a-n}+ \tilde{k}^{-m} \sum_{d=1}^{k} \tilde{d}^{-\frac{4}{3}} \lesssim  \tilde{k}^{-m}.
\end{equation}
By \emph{(\ref{tau f_k C alpha})}, the sequences $\{f_k\}_{k=1}^{\infty}$ and $\{f_{k}^{-1}\}_{k=1}^{\infty}$ converge in the $C^{\alpha}$ norm.
We define $f$ as the limit of $f_{k}$.
Thanks to the uniform convergence, the limit of $f^{-1}_{k}$ is the inverse of $f$.
Thus, $f$ is a homeomorphism.
In addition, $f$ is piecewise affine.
To check this, we see from \emph{(\ref{omegak})} that $f_{k+1} = f_k \chi_{\omega \setminus \omega_k} + g_k \chi_{\omega_k}$ for a certain $g_k : \omega_k \to \R^n$ piecewise affine.
Thus, $f_{k+1} = \sum_{i=1}^k g_i \chi_{\omega_i \setminus \omega_{i+1}}$, so $f = \sum_{i=1}^{\infty} g_i \chi_{\omega_i \setminus \omega_{i+1}}$, which shows that $f$ is piecewise affine.
Moreover, $Df_{k+1} = \sum_{i=1}^k D g_i \chi_{\omega_i \setminus \omega_{i+1}}$, so for any $p \in (1,m)$, thanks to \emph{(\ref{tau nabla g_i})} and \eqref{tau size Om_k},
\begin{align*}
 & \int_{\omega} |Df_{k+1}|^p \leq \sum_{i=1}^k (\tilde{i} +2)^p \left( |\omega_i| - |\omega_{i+1}| \right) \lesssim \sum_{i=1}^k i^p \left( |\omega_i| - |\omega_{i+1}| \right)\\
  & = \sum_{i=1}^k \left[ i^p - (i-1)^p \right] |\omega_i| - k^p |\omega_{k+1}| \lesssim \sum_{i=1}^k i^{p-1} |\omega_i| \lesssim \sum_{i=1}^k i^{p-1-m} \lesssim 1 ,
\end{align*}
which shows that $f \in W^{1,p} (\omega, \R^n)$.

Thus, properties \emph{(\ref{tau f_k border})}, \emph{(\ref{tau f_k C alpha})}, \emph{(\ref{tau f_k dist kI})} and \eqref{tau size Om_k} imply properties \emph{\ref{tau border})}, \emph{\ref{tau C alpha})} and \emph{\ref{tau dist})}.
The equalities $f (\omega) = F \omega$ and $f (\overline{\omega}) = \overline{F \omega}$ are a consequence of \emph{\ref{tau border})}.

Finally, we estimate the integrability of $D f$. 
Given $t>0$, let $k_{1}=\max\{1,\floor{t-\mathcal{R}}\}$.
By \emph{(\ref{tau bound of Dfk})} and \emph{(\ref{omegak})}, we obtain that for all $k \geq k_1$,
\[
 |D f_k| \leq \tilde{k}_1 + 1 \quad \text{in } \omega \setminus \omega_{k_1} ,
\]
so
\[
 |D f| \leq \tilde{k}_1 + 1 \quad \text{in } \omega \setminus \omega_{k_1} .
\]
Therefore,
\[\{x\in \omega :| D f(x)|>t\}\subset\omega_{k_{1}},\]
hence, from (\ref{tau size Om_k}), we obtain
\[\frac{|\{x\in\omega:| D f(x)|>t\}|}{|\omega|}\lesssim \max\{1,t^{-m}\}. \]
Since $\dist(F,\bigcup_{a=0}^{n-m}E_{j,|F|}^{a})<r_{j,|F|}$ we have $|F|>\frac{1}{2}$; therefore \emph{\ref{tau regularity})} follows.
\end{proof}
Next, we construct a laminate that goes from $E_{j}$ to $E_{j+m}$. Again, its proof follows the construction of Section \ref{sect: sketch of the proof}.
\begin{lem}\label{laminate E_j to E_infty U E_j+1}
Let $j\in\N$ and $A\in E_j$.
Then there exist $N \in \N \cap [2, 2^n]$,  
\begin{equation}\label{eq:Pi}
 P_{1}\in E_{j+m} ;Ê\qquad P_{i}\in E_{j+m}\cup\bigcup_{a=0}^{n-m} E_{j+m}^{a} , \quad 2\leq i\leq N ; \qquad \lambda_i \in [0,1], \quad 1 \leq i \leq N
\end{equation}
such that $\nu_{A}:=\sum_{i=1}^N \lambda_{i} \delta_{P_{i}}$ belongs to $\mathcal{L}(\R^{n\times n})$, $\overline{\nu_{A}}=A$, $P_i \neq P_j$ for $i \neq j$,
\begin{equation}\label{eq:li}
 |A-P_{1}|<2^{-j}; \qquad|P_{i}|=|A| , \quad 1\leq i\leq N; \qquad 1-\lambda_{1}\lesssim\frac{2^{-j}}{|A|\max\{1,|A|^{m-1}\}}.
\end{equation}
\end{lem}
\begin{proof}
Since $A\in E_{j}$, there exist $\sigma_{1}\leq \cdots\le\sigma_{n}$ and $Q\in SO(n)$ such that $2^{-j-m}<\sigma_{i} \max\lbrace 1,\sigma_{n}^{m-1}\rbrace<2^{-j}$  for $1\leq i\leq n-m+1$, and $A=Q \diag\left(\sigma_1,\ldots,\sigma_{n}\right)Q^T$.

Since $\rho_{j+m,\sigma_n}= \frac{3 \cdot 2^{-j-2-m}}{\max\lbrace 1,\sigma_{n}^{m-1}\rbrace}$, we have $\rho_{j+m,\sigma_n}<\sigma_{1}$.
As in Lemma \ref{lamlem}, we shall construct families \eqref{eq:Bl} by finite induction on $\ell$.
Let $B_{0,0}=A$, $\lambda_{0,0}=1$ and for $0\leq \ell\leq   n-1$,  $0\leq i\leq 2^{\ell}-1$, we assume $\{B_{\ell,i}\}_{i=0}^{2^{\ell}-1}$ and $\{\lambda_{\ell,i}\}_{i=0}^{2^\ell-1}$ have been defined, $Q^TB_{\ell,j}Q$ is diagonal, $\lambda_{\ell,i}\geq 0$,
equations \eqref{eq:inductionBl} and \eqref{lamlem autovalores de Bli > l} hold, $|B_{\ell,i}|=\sigma_n$, 
\[B_{\ell,0}=Q\diag \left(\underbrace{\rho_{j+m,\sigma_n},\ldots,\rho_{j+m,\sigma_n},}_{\min\{\ell,n-m+1\}} \sigma_{\min\{\ell,n-m+1\}+1},\ldots,\sigma_n\right)Q^{T},\]
and
\[\lambda_{\ell,0}=\prod_{k=1}^{\min\{\ell,n-m+1\}}\frac{\sigma_{n}-\sigma_{k}}{\sigma_{n}-\rho_{j+m,\sigma_n}}.\]
We also assume that if $B_{\ell,i}\notin E_{j+m}$ then
 \[\left(Q^{T}B_{\ell,i}Q\right)_{\alpha,\alpha}\in \{\rho_{j+m,\sigma_n},\sigma_n\}, \qquad \alpha=1,\ldots, \ell.\]
With the above induction hypotheses, we construct $\{B_{\ell+1,i}\}_{i=0}^{2^{\ell+1}-1}$ and $\{\lambda_{\ell+1,i}\}_{i=0}^{2^{\ell+1}-1}$ as follows. For any $0\leq i\leq 2^{\ell}-1$, let
\[B_{\ell+1,2i}= \begin{cases} B_{\ell,i}-Q\diag\left(\underbrace{0,\ldots,0,}_{\ell} \sigma_{\ell+1}-\rho_{j+m,\sigma_n},\underbrace{0,\ldots,0}_{n-\ell-1}\right)Q^{T} &\mbox{if } B_{\ell,i}\notin E_{j+m}, \\ 
B_{\ell,i} & \mbox{if } B_{\ell,i}\in E_{j+m}, \end{cases}\]
\[B_{\ell+1,2i+1}= \begin{cases} B_{\ell,i}-Q\diag\left(\underbrace{0,\ldots,0,}_{\ell} \sigma_{\ell+1}-\sigma_{n},\underbrace{0,\ldots,0}_{n-\ell-1}\right) Q^{T}&\mbox{if } B_{\ell,i}\notin E_{j+m}, \\ 
B_{\ell,i} & \mbox{if } B_{\ell,i}\in E_{j+m}, \end{cases}\]
\[\lambda_{\ell+1,2i}=\begin{cases} \frac{\sigma_{n}-\sigma_{\ell+1}}{\sigma_{n}-\rho_{j+m,\sigma_n}}\lambda_{\ell,i} &\mbox{if } B_{\ell,i}\notin E_{j+m}, \\ 
\lambda_{\ell,i} & \mbox{if } B_{\ell,i}\in E_{j+m}, \end{cases}\]
\[\lambda_{\ell+1,2i+1}=\begin{cases} \frac{\sigma_{\ell+1}-\rho_{j+m,\sigma_n}}{\sigma_{n}-\rho_{j+m,\sigma_n}}\lambda_{\ell,i} &\mbox{if } B_{\ell,i}\notin E_{j+m}, \\ 
0 & \mbox{if } B_{\ell,i}\in E_{j+m}. \end{cases}\]
Hence
\[
 B_{\ell,i} = \frac{\sigma_{n}-\sigma_{\ell+1}}{\sigma_{n}-\rho_{j+m,\sigma_n}} B_{\ell+1,2i} + \frac{\sigma_{\ell+1}-\rho_{j+m,\sigma_n}}{\sigma_{n}-\rho_{j+m,\sigma_n}} B_{\ell+1,2i+1}.
\]
Using that $2^{-j-2m}<\rho_{j+m,\sigma_n}\sigma_n<2^{-j-m}$ and the definition of $E_j$, we have that $B_{\ell,0}\in E_{j+m}$ if and only if $\ell\geq n-m+1$. Hence, if $B_{\ell,0}\in E_{j+m}$,
\[
 \lambda_{\ell+1,0}= \lambda_{\ell,0} = \prod_{k=1}^{n-m+1}\frac{\sigma_{n}-\sigma_{k}}{\sigma_{n}-\rho_{j+m,\sigma_n}} =
 \prod_{k=1}^{\min \{ \ell+1, n-m+1\}} \frac{\sigma_{n}-\sigma_{k}}{\sigma_{n}-\rho_{j+m,\sigma_n}},
\]
whereas if $B_{\ell,0}\notin E_{j+m}$,
\[
 \lambda_{\ell+1,0}= \frac{\sigma_{n}-\sigma_{\ell+1}}{\sigma_{n}-\rho_{j+m,\sigma_n}} \lambda_{\ell,i} = \prod_{k=1}^{\ell+1} \frac{\sigma_{n}-\sigma_{k}}{\sigma_{n}-\rho_{j+m,\sigma_n}} = \prod_{k=1}^{\min \{ \ell+1, n-m+1\}} \frac{\sigma_{n}-\sigma_{k}}{\sigma_{n}-\rho_{j+m,\sigma_n}} .
\]
With this, it is clear that $\{B_{\ell+1,i}\}_{i=0}^{2^{\ell+1}-1}$ and $\{\lambda_{\ell+1,i}\}_{i=0}^{2^{\ell+1}-1}$ satisfy the induction hypotheses.

For $i=1,\ldots, 2^{n}$ define $\hat{\lambda}_{i}=\lambda_{n,i-1}$ and $\hat{P}_{i}=B_{n,i-1}$.
In order to make the $\hat{P}_{i}$ distinct, we let $N \in \N$ and $P_1, \ldots, P_N$ be such that
\[ 
 \{ P_1, \ldots, P_N \} = \{ \hat{P}_1, \ldots, \hat{P}_{2^n} \} , \qquad P_i \neq P_j \ \text{ if } \ i \neq j , \qquad P_1 = \hat{P}_1
\]
and define
\[
 \lambda_i = \sum_{j: \hat{P}_j = P_i} \hat{\lambda}_j , \qquad i = 1, \ldots , N .
\]
We have shown that $\nu_A\in \mathcal{L}(\R^{n\times n})$ and $\overline{\nu_{A}}=B_{0,0} = A$. Now we estimate the distance between $A$ and $P_{1}$:
\begin{align*}
|A-P_1|&=|Q \diag\left(\underbrace{\sigma_1-\rho_{j+m,\sigma_n},\ldots,\sigma_{n-m+1}-\rho_{j+m,\sigma_n}}_{n-m+1},\underbrace{0,\ldots,0}_{m-1}\right) Q^T|\\
&=\sigma_{n-m+1}-\rho_{j+m,\sigma_n}<\sigma_{n-m+1}<2^{-j} .
\end{align*}

To finish the proof it only remains to check the last estimate of \eqref{eq:li}.
Notice that
\[
 \lambda_{1} \geq \hat{\lambda}_1= \prod_{k=1}^{n-m+1}\frac{\sigma_{n}-\sigma_{k}}{\sigma_{n}-\rho_{j+m,\sigma_n}} = \prod_{k=1}^{n-m+1}\left( 1-\frac{\sigma_{k}-\rho_{j+m,\sigma_n}}{\sigma_{n}-\rho_{j+m,\sigma_n}}\right) \geq \prod_{k=1}^{n-m+1}\left( 1-\frac{\sigma_{k}}{\sigma_{n}}\right) .
\]
If $|A|\geq 1$, for $1\leq k\leq n-m+1$ we have $\sigma_{k}\sigma_{n}^{m-1}<2^{-j}$, so
\[
1-\lambda_{1} \leq 1-\prod_{k=1}^{n-m+1}\left( 1-\frac{2^{-j}}{\sigma_{n}^{m}}\right)= 1-\left( 1-\frac{2^{-j}}{\sigma_{n}^{m}}\right)^{n-m+1}\lesssim\frac{2^{-j}}{\sigma_{n}^{m}}=\frac{2^{-j}}{|A|^{m}}.
\]
If, on the other hand, $|A|< 1$, then $\sigma_{k}<2^{-j}$ for $1\leq k\leq n-m+1$, and since $A\in E_j$, we know that $|A|>\frac{1}{2}$, so
\begin{equation*}
 1-\lambda_{1} \leq 1-\prod_{k=1}^{n-m+1}\left( 1-\frac{2^{-j}}{\sigma_{n}}\right)=1-\left( 1-\frac{2^{-j}}{\sigma_{n}}\right)^{n-m+1}\lesssim\frac{2^{-j}}{\sigma_{n}}=\frac{2^{-j}}{|A|}
\end{equation*}
and the proof is finished.
\end{proof}

A variant of Lemma \ref{laminate E_j to E_infty U E_j+1} will be needed.
If, instead of starting from an $A\in E_j$, we begin with the identity matrix, the same proof of Lemma \ref{laminate E_j to E_infty U E_j+1} yields the following result, which will be used in the first step of the construction of the sequence approximating the final homeomorphism of Theorem \ref{th1}.

\begin{lem}\label{laminate I to E_infty U E_j_1}
Given $\alpha\in(0,1)$ and $\delta>0$, let $j_1\in\N$ be as in Lemma \ref{tau lemma}.
Then there exist $N \in \N \cap [2, 2^n]$,
\begin{equation*}
P_{1}\in E_{j_1} ;\qquad P_{i}\in E_{j_1}\cup\bigcup_{a=0}^{n-m} E_{j_1}^{a} , \quad 2\leq i\leq N ; \qquad \lambda_i \in [0,1], \quad 1 \leq i \leq N
\end{equation*}
such that $\nu_{I}:=\sum_{i=1}^{N} \lambda_{i} \delta_{P_{i}}$ belongs to $\mathcal{L}(\R^{n\times n})$, $\overline{\nu_{I}}=I$, $P_i \neq P_j$ for $i \neq j$ and
\begin{equation*}
 |I| \lesssim |P_{i}| \lesssim |I| , \qquad 1\leq i\leq N.
\end{equation*}
\end{lem}
Next, we approximate the laminate of Lemma \ref{laminate E_j to E_infty U E_j+1} by a function.
\begin{lem}\label{h lemma}
Let $A\in E_j$. For any bounded open $\omega\subset\R^n$, $\alpha\in (0,1)$ and $\eta>0$ there exists a piecewise affine homeomorphism $h\in W^{1,1}(\omega, \R^n)\cap C^{\alpha}(\overline{\omega})$ satisfying
\begin{enumerate}[(a)]
\item\label{h lemma border} $h(x)=Ax$ on $\partial \omega$.
\item\label{h lemma C alpha} $\|h-A\|_{C^{\alpha}(\overline{\omega})}<\eta$ and 
$\|h^{-1}-A^{-1}\|_{C^{\alpha}(\overline{A\omega})}<\eta$.
\item\label{h lemma E_j+1} $ D h(x)\in E_{j+m}$ a.e.\ $x\in\omega$.
\item\label{h lemma integral} $\int_{\omega}| D h(x)-A|dx\lesssim 2^{-j}|\omega|$.
\item\label{h omega tilde} There exists an open set $\tilde{\omega}\subset\omega$ such that
\begin{enumerate}[(\ref{h omega tilde}1)]
\item\label{h omega tilde puntual} $| D h(x)-A|\lesssim 2^{-j}$ a.e.\ $x\in \omega\setminus\tilde{\omega}$.
\item\label{h omega tilde medida} $|\{x\in\tilde{\omega}:| D h(x)|>t\}|\lesssim 2^{-j}|\omega|t^{-m}$.
\end{enumerate}
\end{enumerate}
\end{lem}
\begin{proof}
First we build the laminate of Lemma \ref{laminate E_j to E_infty U E_j+1}: 
\[\nu_{A}=\sum_{i=1}^N \lambda_{i} \delta_{P_{i}} \in \mathcal{L}(\R^{n\times n})\]
satisfying $\overline{\nu_{A}}=A$, \eqref{eq:Pi} and \eqref{eq:li}.
Let $\ve>0$ be such that
\[
 \ve< \min \left\{ \frac{1}{2}r_{j,|A|} , \ 2^{-j} - |A - P_1|, \ \frac{1}{2} \min_{2 \leq i \leq N} |P_1 - P_i| \right\}
\]
and
\[r_{j,|A|}<2 r_{j,\mathcal{R}}\qquad \text{for }\mathcal{R}\in (|A|-\ve,|A|+\ve).\]
 Then, Proposition \ref{laminate-homeo} gives a piecewise affine homeomorphism $g:\omega\to A\omega$ satisfying
\begin{enumerate}[1)]
\item $g(x)=Ax$ on $\partial\omega$,
\item $\|g-A\|_{C^{\alpha}(\overline{\omega})}<\frac{\eta}{2}$ and 
$\|g^{-1}-A^{-1}\|_{C^{\alpha}(A\overline{\omega})}<\frac{\eta}{2}$,
\item\label{item:DgPi} $|\{x\in\omega: | D g(x)-P_{i}|<\ve \}|=\lambda_{i}|\omega|$ for $i=1,\ldots, N$.
\end{enumerate}
Let $\tilde{\omega}=\{x\in\omega:  g \text{ is affine in a neighbourhood of } x \text{ and } | D g(x)-P_1|> \ve\}$ and 
\[\hat{\omega}=\left\{x\in\omega: 
  \begin{array}{l}
  g \text{ is affine in a neighbourhood of } x \text{ and} \\
  | D g(x)-P_{i}|< \ve \text{ for some } i\in\{2, \ldots, N\} \text{ and } P_{i}\in\bigcup_{a=0}^{n-m} E_{j+m}^{a}
  \end{array} \right\}.\]
As $\dist ( E_{j+m}, \bigcup_{a=0}^{n-m} E_{j+m}^{a}) > \frac{1}{2}$ and $\ve < \frac{1}{4}$, we have that $\hat{\omega} \subset \tilde{\omega}$.
Note also that $\hat{\omega}$ and $\tilde{\omega}$ are open.
Finally, the choice of $\ve$ was done so that, thanks to \ref{item:DgPi}) the set of $x \in \Om$ such that $| D g(x)-P_1|= \ve$ has measure zero.

Since $g$ is piecewise affine, there exist a family $\{\hat{\omega}_{k}\}_{k\in\N}$ of open sets such that $\hat{\omega} =\bigcup_{k=1}^{\infty}\hat{\omega}_{k}$, and $\hat{P}_k\in\R^{n\times n}$, $b_k\in\R^n$ with $g(x)=\hat{P}_{k}x+b_{k}$ in $\hat{\omega}_{k}$.
Recalling that $|P_i|=|A|$ for $i=1,\ldots, N$, and $||A|-|\hat{P}_{k}||<\ve$, we have
\[
 \dist\left(\hat{P}_{k}, \bigcup_{a=0}^{n-m}E_{j+m}^{a}\right)<\ve<\frac{1}{2}r_{j,|A|}<r_{j,|\hat{P}_{k}|} , \qquad k \in \N .
\]
We define $h$ as the piecewise affine homeomorphism given by Lemma \ref{tau lemma} in each $\hat{\omega}_{k}$ and as $g$ in $\overline{\omega} \setminus\bigcup_{k=1}^{\infty}\hat{\omega}_{k}$. By Lemma \ref{glue homeomorphisms}, $h$ is a homeomorphism, and satisfies \emph{(\ref{h lemma border}), (\ref{h lemma C alpha})} and \emph{(\ref{h omega tilde puntual})}. Property \emph{(\ref{h lemma E_j+1})} comes from \emph{(\ref{tau dist})} in Lemma \ref{tau lemma}.
By \emph{(\ref{tau regularity})} of the same lemma we have
\[\frac{|\{x\in\hat{\omega}_{k}:| D h(x)|>t\}|}{|\hat{\omega}_{k}|}\lesssim |\hat{P}_{k}|^m t^{-m},\qquad t>0.\]
Therefore, using \eqref{eq:li} and that for all $k$ there exists $i$ such that $|\hat{P}_{k}-P_{i}|< \ve$, we have
\[\frac{|\{x\in\hat{\omega}:| D h(x)|>t\}|}{|\hat{\omega}|}\lesssim |A|^m t^{-m}.\]
Now, by \eqref{eq:li} and \ref{item:DgPi}), we have
\[
 \left| \left| Dh \right| - \left| A \right| \right| \leq \left| Dh - P_i \right| = \left| Dg - P_i \right| < \ve \qquad \text{a.e. in } \omega \setminus \hat{\omega}, 
\]
for some $i \in \{ 1, \ldots, N \}$ depending on the point.
Hence $| D h|<|A|+\ve$ a.e.\ in $\tilde{\omega}\setminus\hat{\omega}$. 
Thus, if $t > |A|+\ve$ we have
\[|\{x\in\tilde{\omega}: | D h(x)|>t\}|=|\{x\in\hat{\omega}: | D h(x)|>t\}|\lesssim |A|^m t^{-m}|\hat{\omega}|\leq |A|^m t^{-m}|\tilde{\omega}|,\]
whereas, if $0 < t \leq |A|+\ve$ we get $1\lesssim |A|^{m} t^{-m}$ and
\begin{align*}
|\{x\in\tilde{\omega}: | D h(x)|>t\}|&=|\{x\in\tilde{\omega}\setminus\hat{\omega}: | D h(x)|>t\}|+|\{x\in\hat{\omega}: | D h(x)|>t\}|\\
&\lesssim|\tilde{\omega}\setminus\hat{\omega}|+ |A|^{m} t^{-m}|\hat{\omega}|\lesssim |A|^{m} t^{-m}|\tilde{\omega}|.
\end{align*}
Hence, for all $t>0$,
\begin{equation}\label{h lemma integrability in tilde{omega}}
|\{x\in\tilde{\omega}: | D h(x)|>t\}|\lesssim |A|^{m} t^{-m}|\tilde{\omega}|.
\end{equation}
So using (\ref{h lemma integrability in tilde{omega}}) and 
\begin{equation}\label{h lemma measure of tilde{omega}}
|\tilde{\omega}|=(1-\lambda_1)|\omega|\lesssim \frac{2^{-j}}{|A|\max\{1,|A|^{m-1}\}}|\omega|,
\end{equation}
we have $|\{x\in\tilde{\omega}: | D h(x)|>t\}|\lesssim 2^{-j} t^{-m}|\omega|$, that is, \emph{(\ref{h omega tilde medida})}.

Finally,
\begin{equation}\label{eq:intDh-A}
 \int_{\omega} \left| D h(x)-A \right| dx \leq \int_{\tilde{\omega}} \left| D h(x) \right| dx + \left| A \right| \left|\tilde{\omega}\right| + \int_{\omega\setminus\tilde{\omega}} \left| D h(x)-A \right|dx .
\end{equation}
Using now the common formula for calculating the $L^1$ norm of a function in terms of its distribution function, as well as \eqref{h lemma integrability in tilde{omega}}, we obtain
\begin{align*}
 \int_{\tilde{\omega}} \left| D h(x) \right| dx & = \left[ \int_{0}^{2|A|} + \int_{2|A|}^{\infty} \right] \left| \{x\in\tilde{\omega}: \left| D h(x) \right|>t\} \right| dt \\
 & \lesssim \left| A \right| \left|\tilde{\omega}\right| + \left| A \right|^{m} \left| \tilde{\omega} \right|\int_{2|A|}^{\infty} t^{-m} dt \lesssim \left| A \right| \left|\tilde{\omega}\right| ,
\end{align*}
hence, thanks to \eqref{h lemma measure of tilde{omega}},
\begin{equation}\label{eq:intDh}
 \int_{\tilde{\omega}} \left| D h(x) \right| dx + \left| A \right| \left|\tilde{\omega}\right| \lesssim \left| A \right| \left|\tilde{\omega}\right| \leq 2^{-j} \left| \omega \right| ,
\end{equation}
while \emph{(\ref{h omega tilde puntual})} yields
\begin{equation}\label{eq:intDh-Atilde}
 \int_{\omega\setminus\tilde{\omega}} \left| D h(x)-A \right|dx \lesssim 2^{-j} \left| \omega\setminus\tilde{\omega} \right| \leq 2^{-j} \left| \omega \right| .
\end{equation}
Inequalities \eqref{eq:intDh-A}, \eqref{eq:intDh} and \eqref{eq:intDh-Atilde} show \emph{(\ref{h lemma integral})} and finish the proof.
\end{proof}
The next lemma is the analogous of the previous one when one starts with the identity matrix.
Its proof is similar to that of Lemma \ref{h lemma}, but using Lemma \ref{laminate I to E_infty U E_j_1} instead of Lemma \ref{laminate E_j to E_infty U E_j+1}.

\begin{lem}\label{h lemma prime}
For any $\alpha\in (0,1)$ and $\eta>0$ there exist $j_1 \in \N$ and a piecewise affine homeomorphism $h\in W^{1,1}(\Omega, \R^n)\cap C^{\alpha}(\overline{\Omega},\R^{n})$ satisfying
\begin{enumerate}[(a)]
\item $h(x)=x$ on $\partial \Omega$.
\item $\|h-I\|_{C^{\alpha}(\overline{\Omega})}<\eta$ and 
$\|h^{-1}-I\|_{C^{\alpha}(\overline{\Omega})}<\eta.$
\item $ D h(x)\in E_{j_1}$ a.e.\ $x\in\Omega$.
\item $\int_{\Omega} \left| D h(x)-I \right| dx\lesssim \frac{1}{2}$.
\item $\left| \{x\in\Omega:| D h(x)|>t\} \right| \lesssim \left| \Omega \right| t^{-m}$ for all $t>0$.
\end{enumerate}
\end{lem}

\section{Proof of the theorem}\label{sect: proof of the theorem}

We are in a position to prove Theorem \ref{th1} using Lemmas \ref{h lemma} and \ref{h lemma prime}.
\begin{proof}[Proof of Theorem \ref{th1}]
Let $j_1 \in \N$ be as in Lemma \ref{tau lemma}.
For each $j\in\N$ we will construct a piecewise affine homeomorphism $f_j\in W^{1,1} (\Om, \R^n)\cap C^{\alpha}(\overline{\Om}, \R^n)$ such that
\begin{enumerate}[1)]
\item\label{item:th1} $f_j(x)=x$ on $\partial\Om$.
\item\label{item:th2} $\|f_{j}-f_{j-1}\|_{C^{\alpha}(\overline{\Om})}<2^{-j}\delta$ and 
$\|f_{j}^{-1}-f_{j-1}^{-1}\|_{C^{\alpha}(\overline{\Om})}<2^{-j}\delta$.
\item \label{proof of th Df in Ejm} $ D f_{j}\in E_{j_1+(j-1)m}$ a.e.\
\item\label{item:th5} $\int_{\Om}| D f_j(x)- D f_{j-1}(x)|dx \lesssim 2^{-j} \left| \Om \right|$.
\item \label{proof of th integrability} $|\{x\in \Om:| D f_j(x)|>t\}|\lesssim \left| \Om \right| t^{-m}\prod_{k=0}^{j-1}(1-2^{-k-1})^{-m}(1+2^{-k})$.
\end{enumerate}

The construction of $f_j$ proceeds by induction. Let $f_0 = \id$, which does not satisfy \ref{proof of th Df in Ejm}).
We use Lemma \ref{h lemma prime} to create a piecewise affine homeomorphism $f_1$ such that properties \ref{item:th1}--\ref{proof of th integrability}) hold for $j=1$.

Now suppose we have $f_j$. Since $f_j$ is piecewise affine, for each $i \in \N$ there exist $A_{ij}\in \R^{n\times n}$, $b_{ij}\in \R^n$ and  $\Om_{ij} \subset \Om$ open such that $|\Om\setminus\bigcup_i\Om_{ij}|=0$ and
\[ f_j(x)=A_{ij}x+b_{ij}, \quad x\in\Om_{ij} \]
and, by \ref{proof of th Df in Ejm}), $A_{ij}\in E_{j_1+(j-1)m}$.
On each $\Om_{ij}$ we apply Lemma \ref{h lemma}: there exists a piecewise affine homeomorphism $h_{ij} : \overline{\Om}_{ij} \to A_{ij}\overline{\Om}_{ij}+b_{ij}$ in $W^{1,1}$ and in $C^{\alpha}$ such that
\begin{itemize}
\item $h_{ij} (x) =f_j (x)$ for $x \in \partial \Omega_{ij}$.
\item $\|h_{ij}-f_j\|_{C^{\alpha}(\overline{\Omega_{ij}})} < 2^{-(j+2)}\delta$ and
$\|h_{ij}^{-1}-f_j^{-1}\|_{C^{\alpha}(\overline{\Om}_{ij}+b_{ij})}< 2^{-(j+2)}\delta.$
\item $ D h_{ij}(x)\in E_{j_1+jm}$ a.e.\ $x\in\Omega_{ij}$.
\item $\int_{\Omega_{ij}}| D h_{ij}(x)-A_{ij}| dx\lesssim 2^{-j}|\Omega_{ij}|$.
\item There exists an open set $\tilde{\Omega}_{ij}\subset\Omega_{ij}$ such that
\[
 | D h_{ij}(x) - A_{ij}|\lesssim 2^{-j} \ \text{a.e. } x\in \Omega_{ij}\setminus\tilde{\Omega}_{ij}  \quad \text{and} \quad  \frac{\left| \{x\in\tilde{\Omega}_{ij}:| D h_{ij}(x)|>t\} \right|}{\left| \Omega_{ij} \right|}\lesssim 2^{-j} \, t^{-m} .
\]
\end{itemize}
We define the piecewise affine function $f_{j+1} : \overline{\Om} \to \R^n$ as
\[
 f_{j+1} := \begin{cases}
 f_j (x) & \text{if } x \in \overline{\Omega} \setminus \bigcup_{i\in\N} \Omega_{ij} , \\
 h_{ij} (x) & \text{if } x \in \Omega_{ij} \text{ for some } i \in \N .
 \end{cases}
\]
By Lemma \ref{glue homeomorphisms}, it is homeomorphism and, moreover, properties \ref{item:th1})--\ref{item:th5}) hold for $j+1$.
In addition,
\begin{enumerate}[(a)]
 \item\label{item:tha} $\exists C>0\text{ depending only on $n$ such that } | D f_{j+1}- D f_j|\leq C 2^{-j}$ a.e.\ in $\bigcup_i \left(\Om_{ij} \setminus\tilde{\Om}_{ij} \right)$.
 \item\label{item:thb} $\left| \{x\in \bigcup_{i \in \N}\tilde{\Omega}_i:| D f_{j+1}(x)|>t\} \right| \lesssim 2^{-j}|\Omega|t^{-m}$.
\end{enumerate}
To get property \ref{proof of th integrability}) for $f_{j+1}$ we proceed as follows. Let
\[
 t > \frac{1 + 2^{-1} C}{1+2^{-2}} .
\]
On the one hand, thanks to (\ref{item:tha}) and \ref{proof of th integrability}),
\begin{align*}
 \sum_{i=1}^{\infty} \left| \{x\in \Om_{ij}\setminus\tilde{\Om}_{ij}:| D f_{j+1}(x)|>t\} \right| & \leq \left| \{x\in \Om:| D f_{j}(x)|>t-C 2^{-j}\} \right|\\
 & \lesssim  \left| \Om \right| (t-C2^{-j})^{-m} \prod_{k=0}^{j-1}(1-2^{-k-1})^{-m}(1+2^{-k}) ,
\end{align*}
but $(t-C2^{-j})^{-m} \leq (1-2^{-j-1})^{-m}t^{-m}$, hence
\[
 \sum_{i=1}^{\infty} \left| \{x\in \Om_{ij}\setminus\tilde{\Om}_{ij}:| D f_{j+1}(x)|>t\} \right| \lesssim \left| \Om \right| (1-2^{-j-1})^{-m}t^{-m} \prod_{k=0}^{j-1}(1-2^{-k-1})^{-m}(1+2^{-k}) .
\]
Summing this estimate with that of (\ref{item:thb}) we obtain
\begin{align*}
 \left| \{x\in \Omega:| D f_{j+1}(x)|>t\} \right| & \lesssim \left| \Om \right| t^{-m} \left[ (1-2^{-j-1})^{-m} \prod_{k=0}^{j-1}(1-2^{-k-1})^{-m}(1+2^{-k}) + 2^{-j} \right] \\
 & \leq \left| \Om \right| t^{-m} \prod_{k=0}^{j}(1-2^{-k-1})^{-m}(1+2^{-k}) ,
\end{align*}
so property \ref{proof of th integrability}) holds for $j+1$.
This concludes the construction of $\{ f_j \}_{j \in \N}$ with properties \ref{item:th1})--\ref{proof of th integrability}).
 
As in the proof of Lemma \ref{tau lemma}, property \ref{item:th2}) implies that the sequence $\{f_j\}_{j=1}^{\infty}$ converges in $C^{\alpha}$ to a function $f$ that is a homeomorphism with a $C^{\alpha}$ inverse.
It also shows property \emph{\ref{item:theorem5})} of the statement.
Property \ref{item:th5}), on the other hand, shows that $\{ D f_j\}_{j \in \N}$ converges in $L^1$.
Consequently, $f \in W^{1,1} (\Om, \R^n)$.

Now, for a subsequence $Df_j \to D f$ a.e., so, thanks to the continuity of the singular values (see, e.g., \eqref{eq:singvalues}), we obtain from property \ref{proof of th Df in Ejm}) that $Df(x) \in \Gamma_+$ and $\rank \left( D f (x)\right)<m$ a.e.\ in $\Om$. 
From the convergence  $Df_j \to D f$ in measure and property \ref{proof of th integrability}) we have
\[
 \frac{|\{x\in\Omega:| D f(x)|>t\}|}{|\Omega|}\lesssim t^{-m} \prod_{k=0}^{\infty}(1-2^{-k-1})^{-m}(1+2^{-k}) \lesssim t^{-m},\]
and therefore, $ D f\in L^{m,w}(\Om)$.

Finally, to prove that $f$ is the gradient of a convex function, we apply a standard argument.
We assume, without loss of generality, that $\Om$ is connected; otherwise, we would argue with each connected component.
Take a ball $B$ containing $\overline{\Om}$ and define $\tilde{f} : B \to \R^n$ as $f$ in $\Om$ and the identity outside $\Om$.
Then $\tilde{f} \in W^{1,1} (B, \R^n)$.
Choose a family $\{\eta_{\ve}\}_{\ve>0}$ of standard mollifiers, and define $f_{\ve} := \tilde{f} \ast \eta_{\ve}$ in a ball $B_{\ve} \subset B$ containing $\overline{\Om}$.
Then, $Df_{\ve} (x) \in \Gamma_+$ for all $x \in B_{\ve}$.
Consequently, the differential $1$-form $\alpha_{\ve} : = \sum_{i=1}^n f_{\ve}^i \, dx_i$ defined in $B_{\ve}$ is closed, i.e., $d \alpha_{\ve} = 0$, thanks to the symmetry of $Df_{\ve}$. Here $f_{\ve}^i$ are the components of $f_{\ve}$.
By Poincar\'e's lemma, $\alpha_{\ve}$ is exact, i.e., there exists a smooth function $u_{\ve} : B_{\ve} \to \R$ such that $du_{\ve}= \alpha_{\ve}$, so $\nabla u_{\ve} = f_{\ve}$.
We can take $u_{\ve}$ such that $\int_{\Om} u_{\ve} = 0$.
As the Hessian of $u_{\ve}$ is symmetric positive semidefinite, $u_{\ve}$ is convex.
Now, $f_{\ve} \to f$ in $W^{1,1} (\Om, \R^n)$ as $\ve \to 0$.
Thanks to the Poincar\'e inequality, there exists $u \in W^{2,1} (\Om)$ such that $u_{\ve} \to u$ in $W^{2,1} (\Om)$.
Therefore, $\nabla u = f$.
Moreover, $u$ is convex as a limit of convex functions.
Theorem \ref{th1} is proved.
\end{proof}

In fact, since in our construction the sets $E_j$ approximate planes of dimension $m-1$, our function satisfies $ \rank \left( D f (x)\right)=m-1$ a.e.\ in $\Om$.
One can also see that, in our construction, $f \notin W^{1,m} (B, \R^n)$ for any open $B \subset \Om$.

\section{Sharpness of the result}\label{sect: sharpness}

Theorem \ref{th1} is sharp in the following sense.

\begin{thm}\label{sharpness}
Let $\Omega\subset \R^{n}$ be open and bounded, and $m \in \{2,\ldots,n\}$. 
Let $f : \overline{\Omega} \to \R^n$ be a homeomorphism in $W^{1,p}(\Omega,\R^n)$ such that $f|_{\partial \Omega} = \id|_{\partial \Omega}$.
Assume one of the following:
\begin{enumerate}[a)]
\item\label{optionA} $p = m$ and $f$ is H\"{o}lder continuous.
\item\label{optionB} $p>m$.
\end{enumerate}
Then $\rank(D f (x)) \geq m$ for all $x$ in a subset of $\Omega$ of positive measure.
\end{thm}

To prove this theorem we need to show the validity of the area formula for restrictions to planes of dimension $m$.
Given $f \in W^{1,p} (\Omega, \R^n)$, we define $J_m f (x) = \sqrt{\det D f (x)^T  D f (x)}$ for a.e.\ $x \in \Omega$.
We denote by $\mathcal{H}^m$ the $m$-dimensional Hausdorff measure in $\R^n$; when $\mathcal{H}^m$ acts on subsets of coordinate planes of dimension $m$, it can be identified with the Lebesgue measure in $\R^m$.

\begin{lem}\label{area formula sharp}
Let $f:\overline{\Omega} \rightarrow \R^{n}$ be injective and in $W^{1,p}\left(\Omega, \R^n\right)$.
Assume that one of the alternatives \ref{optionA})--\ref{optionB}) of Theorem \ref{sharpness} holds.
Then for almost every $y\in \R^{n-m}$ and all $\mathcal{H}^{m}$-measurable sets $E \subset \Omega\cap \left( \R^{m}\times \lbrace y\rbrace\right)$,
\[ \int_{E}J_m f \,d\mathcal{H}^{m} = \mathcal{H}^{m}\left( f\left( E\right)\right).\]
\end{lem}
\begin{proof}
By Fubini's theorem, for a.e.\ $y \in \R^{n-m}$, the restriction $f|_{\Omega\cap \left( \R^{m}\times  \lbrace y\rbrace\right)}$ is in $W^{1,p}$ with respect to the $\mathcal{H}^m$ measure.
Fix such a $y$.
A standard approximation theorem (see, e.g., \cite[Corollary 6.6.2]{EvGa92}) shows that there exist sequences $\{ f_j \}_{j \in \N}$ in $C^1 (\R^n, \R^n)$ and $\{ E_j \}_{j \in \N}$ of disjoint $\mathcal{H}^m$-measurable subsets of $E\cap \left( \R^{m}\times \lbrace y\rbrace\right)$ such that 
\[
 f(x) = f_j (x) \quad \text{and} \quad D f(x) = D f_j (x) \qquad \text{for all } x \in E_j \ \text{ and } \ j \in \N
\]
and $\mathcal{H}^m (E \setminus \bigcup_{j=1}^{\infty} E_j) = 0$.
Thus,
\[
 \int_E J_m f \,d\mathcal{H}^{m} = \sum_{j=1}^{\infty} \int_{E_j} J_m f \,d\mathcal{H}^{m} .
\]
Now, for each $j \in \N$, thanks to the area formula for regular maps (see \cite[Theorem 3.2.3]{Federer69}) and the fact that $f$ is injective,
\begin{align*}
 & \int_{E_j} J_m f \, d \mathcal{H}^m = \int_{E_j} J_m f_j \, d \mathcal{H}^m = \int_{\R^n} \# \{ E_j \cap f_j^{-1} (y)\} \, d \mathcal{H}^m (y) \\
 & = \int_{\R^n} \# \{ E_j \cap f^{-1} (y)\} \, d \mathcal{H}^m (y) = \mathcal{H}^m (f (E_j)) .
\end{align*}
Therefore,
\[
 \int_E J_m f \,d\mathcal{H}^{m} = \mathcal{H}^m \Big( f \big( \bigcup_{j=1}^{\infty} E_j \big) \Big) .
\]
Under assumptions \emph{\ref{optionA})} or \emph{\ref{optionB})} of Theorem \ref{sharpness}, $f|_{\Omega\cap \left( \R^{m}\times  \lbrace y\rbrace\right)}$ satisfies the $m$-dimensional Luzin (N) condition, i.e., given $A\subset \Omega\cap \left( \R^{m}\times  \{ y\} \right)$ such that $\mathcal{H}^{m}\left( A\right)=0$ then $\mathcal{H}^{m}\left( f\left( A\right)\right)=0$.
The proof under \emph{\ref{optionA})} is due to \cite[Theorem 1.1]{KoMaZu15} (with $\lambda=0$ in the notation there), while the proof under \emph{\ref{optionB})} is classical \cite{MaMi73}.
In either case,
\[
 \mathcal{H}^m \Big( f \big( E \setminus \bigcup_{j=1}^{\infty} E_j \big) \Big) = 0
\]
and the proof is concluded.
\end{proof}

Now we can prove Theorem \ref{sharpness}.

\begin{proof}[Proof of Theorem \ref{sharpness}]
Suppose, for a contradiction, that $\rank(D f (x)) < m$ for a.e.\ $x \in \Omega$.
Then $J_{m}f=0$ a.e., thanks to the Cauchy--Binet formula.
Then, for a.e.\ $y\in \R^{m}$ we have $J_{m} f=0$ $\mathcal{H}^{m}$-a.e.\ in $\Omega\cap \left(\R^{m}\times \lbrace y\rbrace\right)$ and, by Lemma \ref{area formula sharp},
\begin{equation}\label{eq:Hmf}
 \mathcal{H}^{m}\left( f\left( \Omega\cap\left(\R^{m}\times\lbrace y\rbrace\right)\right)\right) = 0 .
\end{equation}

Let $P_{m}:\R^{n}\rightarrow \R^{m}$ be the projection over the first $m$ coordinates: $P_{m} \left(x_{1},\ldots,x_{n}\right)=\left( x_{1},\ldots,x_{m} \right)$, and, for any $y \in \R^{n-m}$, define the set
$\Omega_{y}=P_{m}\left(\Omega\cap\left(\R^{m}\times\lbrace y\rbrace\right)\right)$
and the function $g_{y}:\Omega_{y}\rightarrow \R^m$
 \[g_{y}(x)=P_{m} \left( f\left(x,y\right) \right).\]
Since $f=\id$ on $\partial \Omega$ and 
$\partial \Omega_{y}=P_{m}\left(\partial \Omega\cap \left(\R^{m}\times\lbrace y\rbrace\right)\right)$,
we have that $g_{y} = \id$ on $\partial \Omega_{y}$.
Using now degree theory, this implies
\[
 \deg \left( g_{y},\Omega_{y}, \cdot \right) = \deg \left( \id,\Omega_{y}, \cdot \right)
\]
and, consequently, $\Omega_{y} \subset g_{y} (\Omega_{y})$ (see, e.g., \cite[Theorem 3.1]{Deimling85}).
Fix $y\in \R^{n-m}$ such that \eqref{eq:Hmf} holds and $\Omega_y \neq \emptyset$.
As $P_{m}$ is $1$-Lipschitz and $\Omega_{y}$ is open, we find that
\[ 
 0 < \mathcal{H}^{m}\left( \Omega_{y}\right) \leq \mathcal{H}^{m}\left( g_{y}\left( \Omega_{y}\right)\right) \leq \mathcal{H}^{m}\left( f\left( \Omega\cap\left(\R^{m}\times\lbrace y\rbrace\right)\right)\right) ,
\]
which contradicts \eqref{eq:Hmf} and completes the proof.
\end{proof}

As can be seen from the proof, in Theorem \ref{sharpness} and Lemma \ref{area formula sharp}, conditions \emph{\ref{optionA})}--\emph{\ref{optionB})} can be replaced by any other assumption implying Luzin's condition (N) in $\R^m$.
As mentioned in the proof, the paper \cite{KoMaZu15} shows some of those conditions.

\section*{Acknowledgements}

The authors have been supported by Project MTM2014-57769-C3-1-P of the Spanish Ministry of Economy and Competitivity and the ERC Starting grant no.\ 307179.
C.M.-C.\ has also been supported by the ``Ram\'on y Cajal'' programme and the European Social Fund.

{\small
\bibliography{Bibliography}

\def\cprime{$'$}
\begin{thebibliography}{10}

\bibitem{AlAm99}
{\sc G.~Alberti and L.~Ambrosio}, {\em A geometrical approach to monotone
  functions in {$\mathbb{R}^n$}}, Math. Z., 230 (1999), pp.~259--316.

\bibitem{AsFaSz08}
{\sc K.~Astala, D.~Faraco, and L.~Sz{\'e}kelyhidi, Jr.}, {\em Convex
  integration and the {$L^p$} theory of elliptic equations}, Ann. Sc. Norm.
  Super. Pisa Cl. Sci. (5), 7 (2008), pp.~1--50.

\bibitem{Ball81}
{\sc J.~M. Ball}, {\em Global invertibility of {S}obolev functions and the
  interpenetration of matter}, Proc. Roy. Soc. Edinburgh Sect. A, 88 (1981),
  pp.~315--328.

\bibitem{BoSzVo13}
{\sc N.~Boros, L.~Sz{\'e}kelyhidi, Jr., and A.~Volberg}, {\em Laminates meet
  {B}urkholder functions}, J. Math. Pures Appl. (9), 100 (2013), pp.~687--700.

\bibitem{Brenier91}
{\sc Y.~Brenier}, {\em Polar factorization and monotone rearrangement of
  vector-valued functions}, Comm. Pure Appl. Math., 44 (1991), pp.~375--417.

\bibitem{Cerny11}
{\sc R.~{\v{C}}ern{\'y}}, {\em Homeomorphism with zero {J}acobian: sharp
  integrability of the derivative}, J. Math. Anal. Appl., 373 (2011),
  pp.~161--174.

\bibitem{Cerny15}
\leavevmode\vrule height 2pt depth -1.6pt width 23pt, {\em Bi-{S}obolev
  homeomorphism with zero minors almost everywhere}, Adv. Calc. Var., 8 (2015),
  pp.~1--30.

\bibitem{CoFaMa05}
{\sc S.~Conti, D.~Faraco, and F.~Maggi}, {\em A new approach to counterexamples
  to {$L^1$} estimates: {K}orn's inequality, geometric rigidity, and regularity
  for gradients of separately convex functions}, Arch. Ration. Mech. Anal., 175
  (2005), pp.~287--300.

\bibitem{CoFaMaMu05}
{\sc S.~Conti, D.~Faraco, F.~Maggi, and S.~M{\"u}ller}, {\em Rank-one convex
  functions on {$2\times 2$} symmetric matrices and laminates on rank-three
  lines}, Calc. Var. Partial Differential Equations, 24 (2005), pp.~479--493.

\bibitem{Dacorogna89}
{\sc B.~Dacorogna}, {\em Direct methods in the calculus of variations}, vol.~78
  of Applied Mathematical Sciences, Springer-Verlag, Berlin, 1989.

\bibitem{DeFi14}
{\sc G.~De~Philippis and A.~Figalli}, {\em The {M}onge-{A}mp\`ere equation and
  its link to optimal transportation}, Bull. Amer. Math. Soc. (N.S.), 51
  (2014), pp.~527--580.

\bibitem{Deimling85}
{\sc K.~Deimling}, {\em Nonlinear functional analysis}, Springer, Berlin, 1985.

\bibitem{DoHeSc14}
{\sc L.~D'Onofrio, S.~Hencl, and R.~Schiattarella}, {\em Bi-{S}obolev
  homeomorphism with zero {J}acobian almost everywhere}, Calc. Var. Partial
  Differential Equations, 51 (2014), pp.~139--170.

\bibitem{EvGa92}
{\sc L.~C. Evans and R.~F. Gariepy}, {\em Measure theory and fine properties of
  functions}, Studies in Advanced Mathematics, CRC Press, Boca Raton, FL, 1992.

\bibitem{Faraco03}
{\sc D.~Faraco}, {\em Milton's conjecture on the regularity of solutions to
  isotropic equations}, Ann. Inst. H. Poincar\'e Anal. Non Lin\'eaire, 20
  (2003), pp.~889--909.

\bibitem{Faraco06}
\leavevmode\vrule height 2pt depth -1.6pt width 23pt, {\em Wild mappings built
  on unbounded laminates}, in Proceedings of the Workshop ``New Developments in
  the Calculus of Variations'', M.~Carozza, L.~D'Onofrio, L.~Greco,
  G.~Moscariello, A.~Passarelli~di Napoli, and C.~Sbordone, eds., vol.~2 of
  Sezione statistico-matematica, Universit{\`a} degli studi del Sannio,
  Edizioni Scientifiche Italiane, 2006, pp.~89--108.

\bibitem{Federer69}
{\sc H.~Federer}, {\em Geometric measure theory}, Die Grundlehren der
  mathematischen Wissenschaften, Band 153, Springer-Verlag New York Inc., New
  York, 1969.

\bibitem{GoGoKa90}
{\sc I.~Gohberg, S.~Goldberg, and M.~A. Kaashoek}, {\em Classes of linear
  operators. {V}ol. {I}}, vol.~49 of Operator Theory: Advances and
  Applications, Birkh\"auser Verlag, Basel, 1990.

\bibitem{Hajlasz93}
{\sc P.~Haj{\l}asz}, {\em Change of variables formula under minimal
  assumptions}, Colloq. Math., 64 (1993), pp.~93--101.

\bibitem{Hencl11}
{\sc S.~Hencl}, {\em Sobolev homeomorphism with zero {J}acobian almost
  everywhere}, J. Math. Pures Appl. (9), 95 (2011), pp.~444--458.

\bibitem{KaKoMa01}
{\sc J.~Kauhanen, P.~Koskela, and J.~Mal{\'y}}, {\em Mappings of finite
  distortion: condition {N}}, Michigan Math. J., 49 (2001), pp.~169--181.

\bibitem{KiKr11}
{\sc B.~Kirchheim and J.~Kristensen}, {\em Automatic convexity of rank-1 convex
  functions}, C. R. Math. Acad. Sci. Paris, 349 (2011), pp.~407--409.

\bibitem{KiKr16}
\leavevmode\vrule height 2pt depth -1.6pt width 23pt, {\em On rank one convex
  functions that are homogeneous of degree one}, Arch. Ration. Mech. Anal., To
  appear (2016), pp.~1--32.

\bibitem{KoMaZu15}
{\sc P.~Koskela, J.~Mal{\'y}, and T.~Z{\"u}rcher}, {\em Luzin's condition ({N})
  and modulus of continuity}, Adv. Calc. Var., 8 (2015), pp.~155--171.

\bibitem{LiMa16}
{\sc Z.~Liu and J.~Mal\'y}, {\em A strictly convex {S}obolev function with null
  {H}essian minors}.
\newblock Preprint available at
  http://msekce.karlin.mff.cuni.cz/ms-preprints/kma-preprints/2015-pap/2015-499.pdf.

\bibitem{MaMi73}
{\sc M.~Marcus and V.~J. Mizel}, {\em Transformations by functions in {S}obolev
  spaces and lower semicontinuity for parametric variational problems}, Bull.
  Amer. Math. Soc., 79 (1973), pp.~790--795.

\bibitem{Muller96}
{\sc S.~M{\"u}ller}, {\em Variational models for microstructure and phase
  transitions}, in Calculus of variations and geometric evolution problems
  ({C}etraro, 1996), vol.~1713 of Lecture Notes in Math., Springer, Berlin,
  1999, pp.~85--210.

\bibitem{MuSv03}
{\sc S.~M{\"u}ller and V.~{\v{S}}ver{\'a}k}, {\em Convex integration for
  {L}ipschitz mappings and counterexamples to regularity}, Ann. of Math. (2),
  157 (2003), pp.~715--742.

\bibitem{Pedregal93}
{\sc P.~Pedregal}, {\em Laminates and microstructure}, European J. Appl. Math.,
  4 (1993), pp.~121--149.

\bibitem{Villani03}
{\sc C.~Villani}, {\em Topics in optimal transportation}, vol.~58 of Graduate
  Studies in Mathematics, American Mathematical Society, Providence, RI, 2003.

\end{thebibliography}
\bibliographystyle{siam}
}

\end{document}